\documentclass[pdflatex,sn-basic, Numbered]{sn-jnl}

\usepackage{graphicx}%
\usepackage{multirow}%
\usepackage{amsmath,amssymb,amsfonts}%
\usepackage{amsthm}%
\usepackage{mathrsfs}%
\usepackage[title]{appendix}%
\usepackage{xcolor}%
\usepackage{textcomp}%
\usepackage{manyfoot}%
\usepackage{booktabs}%
\usepackage{algorithm}%
\usepackage{algorithmicx}%
\usepackage{algpseudocode}%
\usepackage{listings}%

\usepackage{todonotes}

\usepackage{tikz}
\usepackage{tikz-network}

\newcommand{ \PBoutcomeD}{{\sc $\D$-outcome}\ }
\newcommand{ \PBPDS}{{\sc Pairing Dominating Set}\ }

\newcommand{\D}{\mathcal{D}}
\newcommand{\St}{\mathcal{S}}

\tikzstyle{noeud}=[circle,inner sep=2, minimum size =3 pt, line width = 1pt, draw=black, fill=white]
\tikzstyle{petit_noeud}=[circle,inner sep=1.5, minimum size =2.5 pt, line width = 0.75pt, draw=black, fill=white]

\usepackage{thmtools}
\usepackage{thm-restate}

\newtheorem{theorem}{Theorem}[section]
\newtheorem{corollary}[theorem]{Corollary}

\newtheorem{definition}[theorem]{Definition}

\newtheorem{proposition}[theorem]{Proposition}

\newtheorem{lemma}[theorem]{Lemma}

\newtheorem{open}[theorem]{Open Question}
\title[Partition strategies for the Maker-Breaker domination game]{Partition strategies for the Maker-Breaker domination game}

\author[1]{\fnm{Guillaume} \sur{Bagan}}\email{guillaume.bagan@univ-lyon1.fr}

\author[1]{\fnm{Eric} \sur{Duch\^ene}}\email{eric.duchene@univ-lyon1.fr}

\author[2]{\fnm{Valentin} \sur{Gledel}}\email{valentin.gledel@univ-smb.fr}

\author[3]{\fnm{Tuomo} \sur{Lehtilä}}\email{tualeh@utu.fi}

\author[1]{\fnm{Aline} \sur{Parreau}}\email{aline.parreau@univ-lyon1.fr}

\affil[1]{\orgname{Univ Lyon}, \orgdiv{CNRS, INSA Lyon, UCBL, Centrale Lyon, Univ Lyon 2, LIRIS, UMR5205,
F-69622}, \orgaddress{\city{Villeurbanne}, \country{France}}}

\affil[2]{\orgname{Université Savoie Mont Blanc}, \orgdiv{CNRS UMR5127, LAMA}, \orgaddress{\city{Chambéry}, \postcode{F-73000}, \country{France}}}

\affil[3]{\orgdiv{Department of mathematics and statistic}, \orgname{University of Turku}, \orgaddress{\country{Finland}}}

\begin{document}

\abstract{
    The Maker-Breaker domination game is a positional game played on a graph by two players called Dominator and Staller. The players alternately select a vertex of the graph that has not yet been chosen. Dominator wins if at some point the vertices she has chosen form a dominating
set of the graph. Staller wins if Dominator cannot form a dominating set. 
Deciding if Dominator has a winning strategy has been shown to be a PSPACE-complete problem even when restricted to chordal or bipartite graphs. In this paper, we consider strategies for Dominator based on partitions of the graph into basic subgraphs where Dominator wins as the second player.
Using partitions into cycles and edges (also called perfect [1,2]-factors), we show that Dominator always wins in regular graphs and that deciding whether Dominator has a winning strategy as a second player can be computed in polynomial time for outerplanar and block graphs. We then study partitions into subgraphs with two universal vertices, which is equivalent to considering the existence of pairing dominating sets with adjacent pairs. We show that in interval graphs, Dominator wins if and only if such a partition exists. In particular, this implies that deciding whether Dominator has a winning strategy playing second is in NP for interval graphs. We finally provide an algorithm in $n^{k+3}$ for $k$-nested interval graphs (i.e. interval graphs with at most $k$ intervals included one in each other).

}

\maketitle

\section{Introduction}

The Maker-Breaker domination game is played on a graph $G=(V,E)$. The two players, called Dominator  and Staller, claim turn by turn unclaimed vertices of the graph. If at some point the vertices claimed by Dominator form a dominating set of $G$, then Dominator wins. Otherwise, that is, if Staller manages to claim a vertex and all its neighbors, Staller wins.
This game has been defined recently by Duch\^ene, Gledel, Parreau and Renault \cite{MBdomgame}. The game belongs to the larger family of Maker-Breaker positional games played on hypergraphs, that was introduced first by Hales in 1963 \cite{Hales1963}, and later  by Erdös and Selfridge in \cite{erdos}. The famous game {\sc hex} belongs to this family \cite{positionalgames}. In this general context, two players called Maker and Breaker alternately claim a vertex from a given hypergraph. Maker wins if she manages to claim all the vertices of a hyperedge, otherwise Breaker wins. The Maker-Breaker domination game played on a graph $G$ is a particular instance of such games, where the corresponding hypergraph has as its set of hyperedges all the dominating sets of $G$, and Maker corresponds to Dominator. The two reference books \cite{beck,positionalgames} constitute a good introduction to the field of positional games for the uninitiated reader. \\

The main issue concerning positional games is about the computation of the outcome of the game, i.e. who wins whatever the strategy of the opponent is. The input of the problem is generally the (hyper)graph on which the game is played (also called in this paper a {\em game position}, a term that includes graphs where several vertices may already have been claimed). In the case of the Maker-Breaker domination game, it has been proved in \cite{MBdomgame} that there are only three possible outcomes:
\begin{itemize}
    \item $\D$, meaning that Dominator wins whoever starts;
    \item $\St$, meaning that Staller wins whoever starts;
    \item $\mathcal{N}$, meaning that the first player has a winning strategy.
\end{itemize}

A general result about Maker-Breaker games ensures that a player has never interest to miss his turn \cite{beck}. Therefore, an outcome $\D$ corresponds to a game position where Dominator wins when playing second. \\

It has been proved in \cite{MBdomgame} that the computation of the outcome of the domination game is a PSPACE-complete problem in general and also on chordal graphs, but can be done in linear time on trees and cographs. Recall that a problem belongs to PSPACE if it can be solved with a polynomial amount of space. Also recall that PSPACE is a superclass of the class NP, so PSPACE-complete problems are generally considered as very hard problems to deal with. In \cite{mbtotalcubic,mbtotal}, a variation of the domination game has been defined using the total domination property. For these two kinds of domination games, there are several families of graphs (like paths, cycles or particular cases of cycles or grids) for which winning strategies for Dominator can be exhibited by partitioning the vertex set of the graph. More precisely, if the vertices of the graph can be partitioned with partial subgraphs where each outcome is $\D$, then the overall outcome is also $\D$ (see Proposition 3 of \cite{MBdomgame}). Indeed, it suffices for Dominator to apply her winning strategy as a second player on each subgraph. As an illustration, in the standard domination game, cycles and cliques of size at least~$2$ have an outcome $\D$. Therefore, if the game is played on a graph that admits a Hamiltonian cycle or a perfect matching, then the outcome is $\D$. As depicted by Figure~\ref{fig:partitionCliques}, it is also the case if the graph can be partitioned with a union of cliques of size at least $2$. For all these situations, the computation of the outcome and the computation of the winning strategy (i.e. the sequence of winning moves corresponding to this outcome) can both be done in polynomial time. Indeed, the winning strategies for paths, cycles, grids or cliques are easy to describe with pairing arguments. Yet, the two problems are not always correlated in terms of complexity.

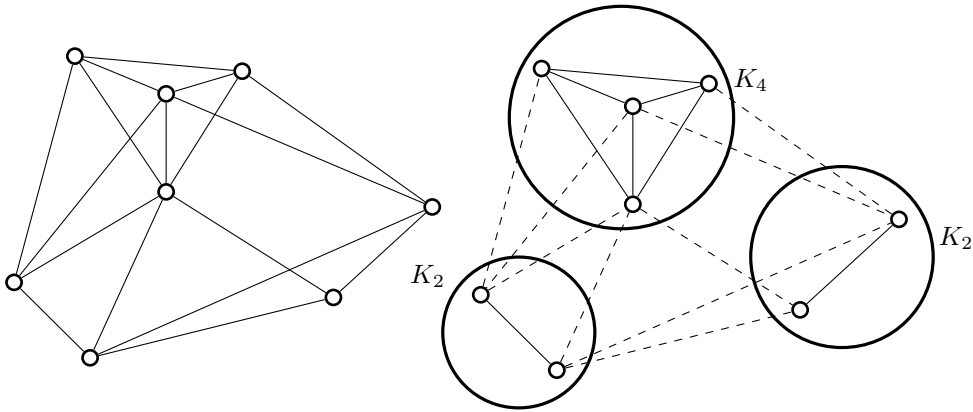
\begin{figure}[ht]
\centering
\begin{tikzpicture}[scale=0.77]

\node at (0,0){
\begin{tikzpicture}
\node[noeud] (a) at (0,1){};
\node[noeud] (b) at (1,0){};
\node[noeud] (c) at (0.8,4){};
\node[noeud] (d) at (2,2.2){};
\node[noeud] (e) at (3,3.8){};
\node[noeud] (f) at (4.2,0.8){};
\node[noeud] (g) at (5.5,2){};
\node[noeud] (h) at (2,3.5){};

\draw (a) -- (b);
\draw (c) -- (d);
\draw (c) -- (e);
\draw (f) -- (g);
\draw (d) -- (e);
\draw (h) -- (c);
\draw (h) -- (d);
\draw (h) -- (e);

\draw (a) -- (c);
\draw (a) -- (d);
\draw (b) -- (d);
\draw (b) -- (f);
\draw (d) -- (f);
\draw (e) -- (g);
\draw (h) -- (g);
\draw (h) -- (a);
\draw (b) -- (g);
\end{tikzpicture}
};

\node at (8,0){
\begin{tikzpicture}
\node[noeud] (a) at (0,1){};
\node[noeud] (b) at (1,0){};
\node[noeud] (c) at (0.8,4){};
\node[noeud] (d) at (2,2.2){};
\node[noeud] (e) at (3,3.8){};
\node[noeud] (f) at (4.2,0.8){};
\node[noeud] (g) at (5.5,2){};
\node[noeud] (h) at (2,3.5){};

\draw (a) -- (b);
\draw (c) -- (d);
\draw (c) -- (e);
\draw (f) -- (g);
\draw (d) -- (e);
\draw (h) -- (c);
\draw (h) -- (d);
\draw (h) -- (e);

\draw[dashed] (a) -- (c);
\draw[dashed] (a) -- (d);
\draw[dashed] (b) -- (d);
\draw[dashed] (b) -- (f);
\draw[dashed] (d) -- (f);
\draw[dashed] (e) -- (g);
\draw[dashed] (h) -- (g);
\draw[dashed] (h) -- (a);
\draw[dashed] (b) -- (g);

\draw[very thick] (0.5,0.5) circle (1);
\draw[very thick] (1.85,3.35) circle (1.475);
\draw[very thick] (4.75, 1.5) circle (1.2);

\node at (-0.7,1.25){$K_2$};
\node at (6.25,1.75){$K_2$};
\node at (3.55,3.85){$K_4$};
\end{tikzpicture}
};

\end{tikzpicture}
\caption{A graph having a partition into three cliques of size at least $2$: the outcome is $\D$.}
\label{fig:partitionCliques}
\end{figure}

The aim of the current paper is to explore the correlation between the existence of such partitions and the existence of a winning strategy for Dominator. In particular, there are cases for which these two problems are proved to be equivalent. For example, in \cite{MBdomgame}, it is proved that the domination game has outcome $\D$ in trees if and only if the graph admits a perfect matching. Such an equivalence is not true on general graphs as in terms of first-order logic, combinatorial structures of graphs and $2$-player games are fundamentally different. In addition, the well-known concept of {\em pairing strategies}~\cite{beck} in positional games is strongly correlated to such partitioning strategies. Pairing strategies have been instantiated in the domination game under the name of {\em pairing dominating sets}. The relation between pairing dominating sets and winning strategies for Dominator is also explored in \cite{MBdomgame}. In terms of algorithmic complexity, equivalences between partition or pairing strategies and the outcome of a game open a door to a reduction of the general complexity of the latter problem. Indeed, the existence of a pairing dominating set in a graph is proved to be a NP-complete problem, whereas the outcome of the domination game is $\mathsf{PSPACE}$-complete in general. The main motivation of the current paper is to investigate such equivalences, as it corresponds to a rare way to decrease the natural complexity of $2$-player games (i.e. $\mathsf{PSPACE}$) to the one of standard combinatorial problems (i.e. $\mathsf{NP}$).\\

The paper is organized as follows. In Section~\ref{sec:tools}, we give the suitable general material to have a formal definition of such partitions of the graph. More precisely, the notion of $\mathcal F$-factor of a graph will be introduced, where the set $\mathcal F$ corresponds to the set of graphs allowed in the partition. In order to yield winning strategies for Dominator in larger classes of graphs, a natural objective is to consider sets of simple or small graphs having outcome $\D$ for $\mathcal F$. As an example, the set of cliques for $\mathcal F$ can be reduced by considering only the graphs with two adjacent universal vertices, that are also $\D$. On Figure~\ref{fig:partitionCliques}, any edge of the $K_4$ can thus be removed without changing the result. 

In Section~\ref{sec:12factor}, the case of graphs having perfect $[1,2]$-factors is developed. They correspond to the case where $\mathcal F$ is the set of edges and cycles. Deciding whether a graph admits such a factor is polynomial. By proving the equivalence between this problem and the computation of the outcome, this leads us to a polynomial time algorithm (both for the outcome and the winning strategy) for the domination game played on a family  of graphs including outerplanar graphs and block graphs. Regular graphs always admit perfect $[1,2]$-factors, leading to a direct outcome $\D$ for them. 

In Section~\ref{sec:interval}, we recall the notion of pairing dominating sets, and see how they can be connected to $\mathcal F$-factors. In particular, if pairing dominating sets have the property that both elements of the pairs are adjacent, then they  correspond exactly to $\mathcal F$-factors, where $\mathcal F$ is the set of complete bipartite graphs, with an additional edge between the two vertices of the first part. We then focus on interval graphs for which this property of adjacency is true. We solve the case of unit interval graphs. Then, the major result of the paper is given: we prove that deciding the existence of a pairing dominating set 
in this class is equivalent to deciding whether the outcome is $\D$. This result
 improves the known complexity of the domination game in interval graphs (i.e. $\mathsf{NP}$). We then investigate the resolution of the pairing dominating set problem in interval graphs and provide an algorithm in $O(n^{k+3})$ for $k$-nested interval graphs that are interval graphs with at most $k$ intervals included  in each other.

\section{Preliminaries}\label{sec:tools}

\subsection{Graphs}

We give here the terminology about graphs that is required for a good understanding of this paper. First, all considered graphs here will be finite and simple. They will generally be described by pairs $(V,E)$ where $V$ is the set of vertices and $E$ the set of edges. We will denote by $N[v]$ the closed neighborhood of a vertex $v$ (i.e. $v$ and all its neighbors). If $X$ is a subset of vertices, we define $N(X)$ as the set of all the vertices that are neighbors to at least one vertex of $X$. Two vertices $x$ and $y$ are {\em twins} if they satisfy $N[x]=N[y]$.

Given a graph $G$ with $(V,E)$ and a subset $V'$ of $V$, we say that a subgraph $G'$ of $G$ is induced by $V'$ if it has $V'$ as set of vertices, and an edge $(x,y)$ belongs to $G'$ if and only if $(x,y)\in E$.

A set $S$ of vertices of a graph $G$ is a {\em dominating set} if for every vertex $v$, there is a vertex $x$ in $S \cap N[v]$. 

An {\em interval graph} is the intersection graph of intervals of the real line, i.e. each vertex corresponds to an interval, and two vertices are adjacent if their corresponding intervals intersect.

A graph is called a {\em cactus graph} if any two cycles of the graph  have at most one vertex in common. A graph is said to be {\em outerplanar} if it is planar and all the vertices belong to the outer face of the planar drawing.

Given a vertex $u$ of a graph $G$, the graph $G-u$ is the subgraph of $G$ where $u$ and all the edges incident to $u$ have been removed. 
A {\em cut-vertex} in a graph is a vertex that disconnects the graph when removed.

A {\em $2$-connected graph} is a connected graph that remains connected after the removal of any vertex. A {\em $2$-connected component} of a graph is a maximal $2$-connected subgraph of it. A graph is called a {\em block graph} if every $2$-connected component is a clique.

A graph is said {\em regular} if all its vertices have the same degree. The term {\em k-regular} is used when this degree is equal to $k$.

We say that $H=(V_H,E_H)$ is a {\em spanning subgraph} of $G=(V_G,E_G)$ if $V_H=V_G$ and $E_H\subseteq E_G$. We say that a spanning subgraph $H$ of $G$ is a {\em perfect matching} if all the vertices of $H$ have degree one. A graph is {\em Hamiltonian} if it admits a connected spanning subgraph where all the vertices have degree two.

\subsection{Maker-Breaker domination game}

Recall that for the Maker-Breaker domination game, there are three possible outcomes, namely $\D$, $\St$ and $\mathcal{N}$. We will use the notation $o(G)$ to define the outcome of a graph $G$. In general, we will consider game positions where all the vertices are unclaimed at the beginning. Yet, there are some situations where it is useful to consider game positions where some moves have already been played by Dominator and/or Staller. Such positions played on a graph $G$ will be denoted by $(G,D,S)$, with $D\cap S=\emptyset$, and where $D$ (resp. $S$) is a set of vertices that are already claimed by Dominator (resp. Staller). Note that a position $G$ simply denotes the position $(G,\emptyset, \emptyset)$.

As explained in the introduction, in Maker-Breaker positional games, both players never have interest to miss their turn. Therefore, if a player has a winning strategy as a second player, then he also wins being first. As the current paper deals with strategies for Dominator, we will mainly identify some of the situations where she always wins (being first or second) or not. The previous remark leads us to consider only the outcome where Dominator is the second player. Therefore, for the study we are concerned with, the decision problem related to the domination game will be defined as follows:\\

\noindent \PBoutcomeD~\\
{\bf Input:} A graph $G$ \\
{\bf Output:} Does $G$ satisfy $o(G)=\D$ ?\\

It is known from~\cite{MBdomgame} that \PBoutcomeD is a PSPACE-complete problem.
The next two propositions are easy observations that constitute the starting point for partition strategies. 

\begin{proposition}[\cite{MBdomgame}]\label{prop:union}
    Let $G_1$ and $G_2$ be two graphs on disjoint sets of vertices with outcome $\D$. Then the disjoint union of $G_1$ and $G_2$, denoted by $G_1\cup G_2$, has outcome $\D$.
\end{proposition}


\begin{proposition}[Monotonicity \cite{MBdomgame}]\label{prop:monotonicity}
  Let $G$ be a graph and $H$ be a spanning subgraph of $G$. If $H$ has outcome $\D$, then $G$ also has outcome $\D$.
\end{proposition}

The next definition has been introduced by the authors of \cite{MBdomgame} and illustrates a graph structure that guarantees the existence of a winning strategy for Dominator. It will be a key element of Section~\ref{sec:interval} dedicated to interval graphs.

\begin{definition}[\cite{MBdomgame}]\label{def:PDS}
    Let $G=(V,E)$ be a graph. A set of pairs of vertices $\{(u_1,v_1), . . . , (u_k, v_k)\}$ of $V$ is a {\em pairing dominating set} (PDS for short) if all the vertices in the pairs are distinct and if the intersection of the closed neighborhoods of each pair covers the vertices of the
graph:
$$V = \cup_{i=1}^k N[u_i]\cap N[v_i]$$
\end{definition}

Figure~\ref{fig:pds-example} gives an example of a pairing dominating set.

\begin{figure}[h]
    \centering
    \begin{tikzpicture}
        \node[noeud] (u1) at (0,1){};
        \node[noeud] (v1) at (0,-1){};
        \node[noeud] (u3) at (1,0){};
        \node[noeud] (a) at (2,-1){};
        \node[noeud] (b) at (2,0){};
        \node[noeud] (c) at (2,1){};
        \node[noeud] (v3) at (3,0){};
        \node[noeud] (u2) at (4,1){};
        \node[noeud] (v2) at (4,-1){};

        \draw (u1) -- (v1) -- (u3) -- (u1);
        \draw (u2) -- (v2) -- (v3) -- (u2);
        \draw (u3) -- (a) -- (v3);
        \draw (u3) -- (b) -- (v3);
        \draw (u3) -- (c) -- (v3);

        \node[left = 2pt of u1] {$u_1$};
        \node[left = 2pt of v1] {$v_1$};
        \node[right = 2pt of u2] {$u_2$};
        \node[right = 2pt of v2] {$v_2$};
        \node[above = 2pt of u3] {$u_3$};
        \node[above = 2pt of v3] {$v_3$};
        
    \end{tikzpicture}
    \caption{The set $\{(u_1,v_1),(u_2,v_2),(u_3,v_3)\}$ is a pairing dominating set.}
    \label{fig:pds-example}
\end{figure}
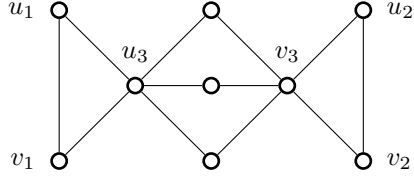

Having a pairing dominating set is a sufficient condition for a graph to have outcome $\D$. Indeed, Dominator playing second can claim at least one vertex of each pair and thus dominate the graph.

\begin{proposition}[\cite{MBdomgame}]
    Let $G$ be a graph. If $G$ has a pairing dominating set, then $G$ has outcome~$\D$.
\end{proposition}

The next lemma is a consequence of the general super Lemma 1.84 from Oijid's thesis~\cite{nacimthesis} in the context of the Maker-Breaker domination game: 

\begin{lemma}[Super lemma \cite{nacimthesis}]
\label{superlemma}
Let $(G,D,S)$ be any game position, and let $(x,y)$ be any pair of unclaimed vertices. If $x$ and $y$ are twins, then the positions $(G,D,S)$ and $(G,D\cup \{x\}, S\cup \{y\})$ have the same outcome. 
\end{lemma}

As a consequence of this result, from a position of outcome $\D$, if a vertex $w$ of the graph has its neighborhood included in the neighborhood of another vertex $v$, then one can assume that Dominator keeps her winning strategy by answering $v$ (if available) after a move of Staller in $w$.

\begin{lemma}\label{BadPlacementLemma}
Let $G$ be a position with $o(G)=\D$ and let us have two vertices $w,v\in V(G)$ with $N[w]\subseteq N[v]$. There also exists a winning strategy for Dominator playing second in $(G,\{v\},\{w\})$.
\end{lemma}

\begin{proof}
    Consider the supergraph $G'$ of $G$ with the same set of vertices, by adding all the necessary edges incident to $w$ so that $v$ and $w$ are now twins. By application of Lemma~\ref{superlemma}, we have that Dominator wins on $(G',\{v\},\{w\})$. In this graph, since the vertex $v$ is claimed by Dominator, all the edges that have been added have their two extremities dominated by $v$. Therefore, Dominator also wins playing second in $(G,\{v\},\{w\})$.
\end{proof}





\subsection{$\mathcal F$-factors}

Propositions \ref{prop:union} and \ref{prop:monotonicity} allow us to define a type of strategy for Dominator based on a decomposition of the graph into small graphs of outcome $\D$. We formalize these strategies using the notion of $\mathcal F$-factors (see for example \cite{factor,HELL,hell2}).

\begin{definition}
Let $\mathcal F$ be a family of graphs. Let $G$ be a graph.
    An {\em $\mathcal F$-factor} of $G$ is a spanning subgraph $H$ such that each connected component of $H$ is isomorphic to some graph in $\mathcal{F}$.
\end{definition}

\begin{proposition}\label{prop:factorD}
Let $G$ be a graph. If $G$ has an $\mathcal F$-factor where all the elements of $\mathcal F$ have outcome $\D$, then $G$ has outcome $\D$.
\end{proposition}

\begin{proof}
Let $\mathcal F$ be a family of graphs of outcome $\D$ and $H$ be an $\mathcal F$-factor of $G$. Since $H$ is a union of graphs of outcome $\D$, by Proposition \ref{prop:union}, it has outcome $\D$. Since it is a spanning subgraph of $G$, by Proposition \ref{prop:monotonicity}, $G$ has outcome $\D$.
\end{proof}

Depending on the choice of $\mathcal F$, we obtain different strategies for Dominator.
For example, if one takes for $\mathcal F$ the smallest graph of outcome $\D$, the graph $K_2$, then a $\{K_2\}$-factor is simply a perfect matching. Proposition \ref{prop:factorD} says that if a graph has a perfect matching, then it has outcome $\D$ \cite{MBdomgame}.
Other simple graphs have outcome $\D$ like any (odd) cycle (see \cite{MBdomgame}). Combining $K_2$ and odd cycles we obtain {\em perfect $[1,2]$-factors} that will be  investigated more in Section \ref{sec:12factor}.

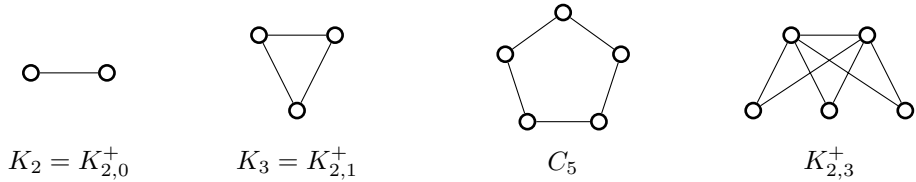
\begin{figure}[h]
    \centering
    \begin{tikzpicture}

 \begin{scope}[shift={(0,0)}]
    \node[noeud](x) at (0,-0.5) {};
    \node[noeud](y) at (1,-0.5) {};
  
    \draw (x)--(y);

    \node at (0.5,-1.7) {$K_2=K^+_{2,0}$};
    \end{scope}

 \begin{scope}[shift={(3,0)}]
    \node[noeud](x) at (0,0) {};
    \node[noeud](y) at (1,0) {};
    \node[noeud](1) at (0.5,-1) {};

    \draw (x)--(y) --(1)--(x);

    \node at (0.5,-1.7) {$K_3=K^+_{2,1}$};
    \end{scope}

     \begin{scope}[shift={(7,-0.5)}]
    \node[noeud](x) at (18:0.8) {};
    \node[noeud](y) at (90:0.8) {};
    \node[noeud](1) at (162:0.8) {};
  \node[noeud](2) at (234:0.8) {};
  \node[noeud](3) at (306:0.8) {};

    \draw (x)--(y) --(1)--(2)--(3)--(x);

    \node at (0,-1.2) {$C_5$};
    \end{scope}
    
    \begin{scope}[shift={(10,0)}]
    \node[noeud](x) at (0,0) {};
    \node[noeud](y) at (1,0) {};
    \node[noeud](1) at (-0.5,-1) {};
    \node[noeud](2) at (0.5,-1) {};
    \node[noeud](3) at (1.5,-1) {};

    \draw (x)--(y) --(1)--(x)--(2)--(y)--(3)--(x);

    \node at (0.5,-1.7) {$K^+_{2,3}$};
    \end{scope}
    
    \end{tikzpicture}
    \caption{Small graphs of outcome $\D$ used for covering strategies.}
    \label{fig:smallD}
\end{figure}

We will consider in Section \ref{sec:interval} another family of graphs of outcome $\D$ that permits us to express the pairing strategies used in \cite{MBdomgame} in our framework. Let $\ell$ be an integer. We define the graph $K^+_{2,\ell}$ as the complete bipartite graph $K_{2,\ell}$ with an additional edge between the two vertices in the part of size 2. By extension, $K^+_{2,0}$ denotes $K_2$. See Figure \ref{fig:smallD} for an illustration. Since this graph has two universal vertices, Dominator can take one of the universal vertices at her first turn and thus it has outcome $\D$. We will see in Section~\ref{sec:interval} that a $\{K^+_{2,\ell}, \ell\geq 0\}$-factor actually corresponds to a pairing dominating set where the pairs of vertices are adjacent. We will in particular prove that an interval graph $G$ has outcome $\D$ if and only if it has such a factor.

Note that if one can compute in polynomial time a strategy for Dominator as second player for each element of $\mathcal F$, then once we have an $\mathcal F$-factor, one can find a winning strategy for Dominator in polynomial time: it suffices to follow the strategy in the component where Staller has played. This will be the case for all the graphs of outcome $\D$ which we are considering in this paper.

\section{Perfect $[1,2]$-factors}\label{sec:12factor}

In this section we focus on covering strategies with edges and cycles, i.e. on $\mathcal F$-factors with $\mathcal F=\{K_2\}\cup\{C_n,n\geq 3\}$. Such factors are also called
{\em perfect $[1,2]$-factors} \cite{factor}. 
We first give a characterisation of graphs admitting perfect $[1,2]$-factors using perfect matchings in a special bipartite graph \cite{Tutte}. This characterisation has two consequences. First, deciding if a graph has a perfect $[1,2]$-factor and exhibiting one if it exists is polynomial. Second, regular graphs always have a perfect $[1,2]$-factor, and thus outcome $\D$.
Then we prove that having a perfect $[1,2]$-factor is equivalent to outcome $\D$ in a class of graphs that contains block and outerplanar graphs. As a consequence, deciding if a graph has outcome $\D$ can be computed in polynomial time in these classes.

\subsection{Perfect $[1,2]$-factors as matchings}

A {\em $[1,2]$-factor} is a spanning subgraph where all the vertices have degree 1 or 2, i.e. all the connected components are cycles or paths. It is {\em perfect} if all the connected components are regular. Thus a perfect $[1,2]$-factor must have all its connected components isomorphic to an edge or a cycle. This is exactly a $\{K_2,C_n:n\geq 3\}$-factor.
Since any cycle has outcome $\D$ \cite{MBdomgame}, by Proposition~\ref{prop:factorD}, if a graph $G$ has a perfect $[1,2]$-factor, then it has outcome $\D$.

\begin{lemma}\label{LemPer12fact}
    Let $G$ be a graph that admits a perfect $[1,2]$-factor. Then $o(G)=\D$.
\end{lemma}

Tutte \cite{Tutte} gives a characterisation of graphs having a perfect $[1,2]$-factor using the incidence bipartite graph $B(G)$ of $G$: the vertices of $B(G)$ are two copies $V_1$ and $V_2$ of the vertices of $G$. There is an edge between $u_1\in V_1$ and $u_2\in V_2$ if and only if the vertices corresponding to $u_1$ and $u_2$ are adjacent in $G$.
 A nice proof of the following result can be found in \cite{factor}.

 \begin{theorem}\cite{factor,Tutte}\label{TheTutte}
     Let $G$ be a graph. Then $G$ has a perfect $[1,2]$-factor if and only if $B(G)$ has a perfect matching. Moreover, one can find the perfect $[1,2]$-factor of $G$ from the perfect matching of $B(G)$ in polynomial time.
 \end{theorem}

This theorem has two important consequences. First, since deciding if a graph has a perfect matching is polynomial (using for example Ford-Fulkerson algorithm, \cite{ford1956maximal}), deciding if a graph has a perfect $[1,2]$-factor is also polynomial.

\begin{corollary}\label{Cor12FactorPol}
 Given a graph $G$, deciding if $G$ has a perfect $[1,2]$-factor is polynomial. Moreover, if $G$ has such a factor, then it can be computed in polynomial time.
\end{corollary}

Second, using Hall's theorem, there is a sufficient condition on the neighborhoods of $B(G)$ for $G$ to have a perfect $[1,2]$-factor.

\begin{theorem}[Hall's theorem \cite{hall}]\label{TheHall}
    Let $B=(X\cup Y,E)$ be a bipartite graph with bipartition to $X$ and $Y$.
    There is a perfect matching in $B$ if and only if $|X|=|Y|$ and, for any subset $X'$ of $X$, $|N(X')|\geq |X'|$.
\end{theorem}

In our setting, it means that  there is a perfect $[1,2]$-factor in $G$ if and only if for every subset of vertices $X$, the number of neighbors of vertices in $X$ is at least $|X|$.
This is in particular the case for regular graphs. 

\begin{corollary}
    Let $G$ be an $r$-regular graph with $r\geq 1$. Then $G$ has a perfect $[1,2]$-factor and thus, has outcome $\D$. Moreover, the winning strategy for Dominator can be computed in polynomial time.
 \end{corollary}

\begin{proof}
    If $G$ is $r$-regular, then, for any set of vertices $X$, we have  $|N(X)|\geq|X|$. Indeed, there are $r$ edges adjacent to each vertex in $X$ and hence, the multiset of neighbors of vertices in $X$ has $r|X|$ vertices. Since $G$ is $r$-regular, each vertex in the multiset can be adjacent to at most $r$ vertices in $X$. Thus, $|N(X)|\geq|X|$. Hence, $B(G)$ satisfies Hall's condition and thus has a perfect matching. Hence, $G$ admits a perfect $[1,2]$-factor by Theorem \ref{TheTutte} that can be computed in polynomial time by Corollary \ref{Cor12FactorPol}. In particular, $G$ has outcome $\D$ by Lemma \ref{LemPer12fact} and the winning strategy for Dominator can be computed in polynomial time.  
\end{proof}

Note that this result is tight in the sense that there exist graphs with $\Delta(G) = \delta(G) + 1$  such that Dominator loses as the second player. See for example the graphs of Figure~\ref{fig:delta+1}.

\begin{figure}[h]
    \centering
    \begin{tikzpicture}
        \node at (0,0){
            \begin{tikzpicture}
                \node[noeud] (u) at (0,0){};
                \node[noeud] (v) at (1,0){};
                \node[noeud] (w) at (2,0){};
                \draw (u) -- (v) -- (w);
            \end{tikzpicture}
        };

        \node at (0,-1){$G_1$};
        
        \node at (5,0){
            \begin{tikzpicture}
                \node[noeud] (u1) at (-1,0){};
                \node[noeud] (u2) at (-1.7,0.7){};
                \node[noeud] (u3) at (-2.4,0){};
                \node[noeud] (u4) at (-1.7,-0.7){};
                \node[noeud] (w) at (0,0) {};
                \node[noeud] (v1) at (1,0){};
                \node[noeud] (v2) at (1.7,0.7){};
                \node[noeud] (v3) at (2.4,0){};
                \node[noeud] (v4) at (1.7,-0.7){};

                \draw (u1) -- (u2) -- (u3) -- (u4) -- (u1) -- (w) -- (v1) -- (v2) -- (v3) -- (v4) -- (v1);
            \end{tikzpicture}
        };
        
        \node at (5,-1){$G_2$};
    \end{tikzpicture}
    \caption{Examples of graphs for which $\Delta = \delta + 1$ and Dominator loses as the second player.}
    \label{fig:delta+1}
\end{figure}
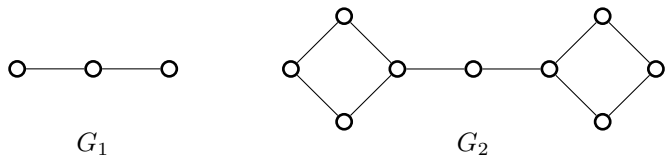

\subsection{The cut-factor property}

We say that a subgraph $F$ of $G$ is a \emph{partial perfect $[1,2]$-factor}  if every component of $F$ is $1$- or $2$-regular. We say that a partial perfect $[1,2]$-factor is {\em maximal} if there does not exist another partial perfect $[1,2]$-factor $F'$ of $G$ such that $|V(F)|<|V(F')|$. Furthermore, a subgraph $F$ of $G$ is a \emph{cut-factor} if it is a partial perfect $[1,2]$-factor such that it is maximal but not perfect and 
for some cut-vertex  $u\in  V(F)$ with $w\in N_F(u)$, there exists a vertex $v\in N(u)\setminus V(F)$ of $G$ such that $G-u$ has vertex $v$ in some component $G_v$ and vertex $w$ in another component $G_w$. In this case, vertex $u$ is called \emph{cut-factor vertex} of $F$. See Figure \ref{fig:exemple_cut-factor} for an example of a cut-factor.
We say that  a class of graphs $\mathcal{G}$  {\em satisfies the cut-factor property} if every graph 
$G \in \mathcal{G}$ either has an isolated vertex, or admits a perfect $[1,2]$-factor or has a cut-factor $F$.

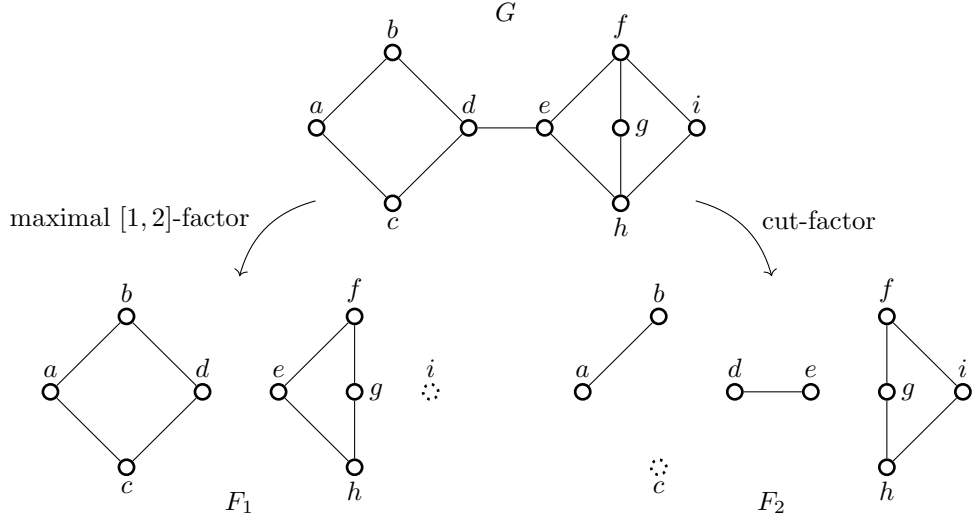
\begin{figure}[h]
    \centering
    \begin{tikzpicture}
        \node at (0,1.5){$G$};
        \node at (0,0){
        \begin{tikzpicture}
            \node[noeud](a) at (0,0){};
            \node[above = 2pt] at (a){$a$};
            \node[noeud](b) at (1,1){};
            \node[above = 2pt] at (b){$b$};
            \node[noeud](c) at (1,-1){};
            \node[below = 2pt] at (c){$c$};
            \node[noeud](d) at (2,0){};
            \node[above = 2pt] at (d){$d$};
            \node[noeud](e) at (3,0){};
            \node[above = 2pt] at (e){$e$};
            \node[noeud](f) at (4,1){};
            \node[above = 2pt] at (f){$f$};
            \node[noeud](g) at (4,0){};
            \node[right = 2pt] at (g){$g$};
            \node[noeud](h) at (4,-1){};
            \node[below = 2pt] at (h){$h$};
            \node[noeud](i) at (5,0){};
            \node[above = 2pt] at (i){$i$};

            \draw (d) -- (b) -- (a) -- (c) -- (d) -- (e) -- (f) -- (g) -- (h) -- (e);
            \draw (f) -- (i) -- (h);
        \end{tikzpicture}
        };

        \draw[->] (-2.5,-1) to [bend right] (-3.5,-2);
        \node[anchor = east] at (-3.25,-1.25){maximal $[1,2]$-factor}; 
        \draw[->] (2.5,-1) to [bend left] (3.5,-2);
        \node[anchor = west] at (3.25,-1.25){cut-factor};

        \node at (-3.5,-5){$F_1$};
        \node at (-3.5,-3.5){
        \begin{tikzpicture}
            \node[noeud](a) at (0,0){};
            \node[above = 2pt] at (a){$a$};
            \node[noeud](b) at (1,1){};
            \node[above = 2pt] at (b){$b$};
            \node[noeud](c) at (1,-1){};
            \node[below = 2pt] at (c){$c$};
            \node[noeud](d) at (2,0){};
            \node[above = 2pt] at (d){$d$};
            \node[noeud](e) at (3,0){};
            \node[above = 2pt] at (e){$e$};
            \node[noeud](f) at (4,1){};
            \node[above = 2pt] at (f){$f$};
            \node[noeud](g) at (4,0){};
            \node[right = 2pt] at (g){$g$};
            \node[noeud](h) at (4,-1){};
            \node[below = 2pt] at (h){$h$};
            \node[noeud, dotted](i) at (5,0){};
            \node[above = 2pt] at (i){$i$};

            \draw (d) -- (b) -- (a) -- (c) -- (d);
            \draw (e) -- (f) -- (g) -- (h) -- (e);
        \end{tikzpicture}
        };

        \node at (3.5,-5){$F_2$};
        \node at (3.5,-3.5){
        \begin{tikzpicture}
            \node[noeud](a) at (0,0){};
            \node[above = 2pt] at (a){$a$};
            \node[noeud](b) at (1,1){};
            \node[above = 2pt] at (b){$b$};
            \node[noeud, dotted](c) at (1,-1){};
            \node[below = 2pt] at (c){$c$};
            \node[noeud](d) at (2,0){};
            \node[above = 2pt] at (d){$d$};
            \node[noeud](e) at (3,0){};
            \node[above = 2pt] at (e){$e$};
            \node[noeud](f) at (4,1){};
            \node[above = 2pt] at (f){$f$};
            \node[noeud](g) at (4,0){};
            \node[right = 2pt] at (g){$g$};
            \node[noeud](h) at (4,-1){};
            \node[below = 2pt] at (h){$h$};
            \node[noeud](i) at (5,0){};
            \node[above = 2pt] at (i){$i$};

            \draw (a) -- (b);
            \draw (d) -- (e);
            \draw (f) -- (i) -- (h) -- (g) -- (f);
        \end{tikzpicture}
        };
        
    \end{tikzpicture}
    \caption{Both $F_1$ and $F_2$ are maximal $[1,2]$-partial factors of $G$. However, only $F_2$ is a cut-factor of $G$ as $i$, the only vertex not in $F_1$, is not adjacent to a cut-vertex in $G$ but $c$, the only vertex not in $F_2$, is adjacent to $d$ which is a cut-vertex. Moreover, in $G-d$ the vertex $e$, the neighbor of $d$ in $F_2$, is not in the same connected component as $c$.}
    \label{fig:exemple_cut-factor}
\end{figure}

To illustrate the notion of cut-factor, consider cut-factors in trees. If a tree $T$ does not admit a perfect $[1,2]$-factor (i.e a perfect matching) and is not an isolated vertex, it admits a cut-factor. Indeed, consider a maximal matching $M$ in $T$. When $M$ is not perfect, consider $x$ that is not covered by $M$. Then any vertex $u$ adjacent to $x$ must be covered by $M$ and be a cut-factor vertex since in $T-u$ the vertex paired with $u$ in $M$ cannot be in the same component than $x$. Hence, trees have the cut-factor property. Moreover, they are also closed by induced subgraphs. Thus, if we consider the following lemma for trees, then we obtain the known result that we have $o(T)=\D$ if and only if $T$ has a perfect matching  \cite{MBdomgame}.

\begin{lemma}\label{LemPVMP}
    Let $\mathcal{G}$ be a class of graphs, closed by induced subgraphs, that satisfies the cut-factor property.   
    Then, for any graph $G\in \mathcal{G}$, we have $o(G)=\D$ if and only if $G$ admits a perfect $[1,2]$-factor.
\end{lemma}

\begin{proof}
    Observe first that if a graph $G$ admits a perfect $[1,2]$-factor, then $o(G)=\D$ by Lemma~\ref{LemPer12fact}. Secondly, if $G$ has an isolated vertex, then it does not admit a perfect $[1,2]$-factor and Staller wins as the first player. Thus, we will consider graphs of $\mathcal{G}$ which have a cut-factor and prove that Staller wins being first in these graphs.

   We first observe that the only graph on three vertices that admits a cut-factor is the $3$-path $P_3$. In $P_3$, Staller wins playing first by claiming the cut-vertex.
   
    We prove the claim by induction on the number of vertices. Assume that the claim is true for every graph on strictly less than $n>3$ vertices. Let $G\in \mathcal{G}$ be a graph on $n>3$ vertices which admits a cut-factor $F$ and has cut-factor vertex $u$. 
    Furthermore, we denote by $w$ a vertex in $N_F(u)$ and by $v$ a vertex in $N(u)\setminus V(F)$ such that in $G-u$ vertices $v$ and $w$ belong to different connected components denoted by $G_v$ and $G_w$, respectively. Vertices $v$ and $w$ exist since $u$ is a cut-factor vertex. 

    Consider the component $G_v$. Assume that it contains a perfect $[1,2]$-factor $F_v$. However, now $F_v$ together with components of $F$ outside of $G_v$ form a partial perfect $[1,2]$-factor  $F_v'$ with $|V(F_v')|>|V(F)|$. Thus, $F$ is not a partial maximum perfect $[1,2]$-factor of $G$, a contradiction. Thus, $G_v$ does not admit a perfect $[1,2]$-factor.

    Consider next component $G_w$. Assume that it contains a perfect $[1,2]$-factor $F_w$. However, now $F_w$ together with $K_2$-component on vertices $u$ and $v$ and components of $F$ outside of $G_w$ form a partial perfect $[1,2]$-factor  $F_w'$ with $|V(F_w')|>|V(F)|$. Thus, $G_w$ does not admit a perfect $[1,2]$-factor.

     Consider the following strategy for Staller: he starts by claiming $u$. Since neither $G_w$ nor $G_v$ admits a perfect $[1,2]$-factor and since $G_v,G_w \in \mathcal G$, by induction, Staller wins as the first player in both $G_w$ and $G_v$. Thus, if Dominator plays in $G_w$ (respectively $G_v$) after Staller has claimed $u$, then Staller wins by following his strategy as the first player on $G_v$ (resp. $G_w$). 
\end{proof}

\subsection{Consequences for outerplanar and block graphs}

In the following lemma, we present a large class of graphs which has the cut-factor property.  Later we prove that it contains outerplanar and block graphs. Note that we cannot immediately apply Lemma \ref{LemPVMP} to the graph class of the following lemma since it is not closed under induced subgraphs.

\begin{lemma}\label{LemHamiltonianCom}
    Let $\mathcal{G}$ be a class of graphs such that every $2$-connected component on at least three vertices is Hamiltonian. Then graph class $\mathcal{G}$ has the cut-factor property.
\end{lemma}

\begin{proof}
    Let $\mathcal{G}$ be as in the claim.  Assume in contrary that there exists a graph $G\in \mathcal{G}$ which does not have an isolated vertex, does not admit a perfect $[1,2]$-factor and does not admit a cut-factor. In the following, we assume that $G$ is connected. We may assume this since if components separately admit  a cut-factor, then the union of those components also admits a cut-factor.
    Let us denote the $2$-connected components of $G$ with at least three vertices by $A_1,\dots, A_h$. We have $h\geq1$ since otherwise $G$ is a tree. Furthermore, for each $i$, let $C(A_i)$ be a Hamiltonian cycle of $A_i$. We call edges between two vertices of $C(A_i)$ \textit{chords} when they do not belong to the cycle $C(A_i)$. 
    Observe that if $F$ is a partial maximum perfect $[1,2]$-factor of $G$, then it remains partial maximum perfect $[1,2]$-factor of $G$ when we transform every even cycle of $F$ into a set of $K_2$-components. Thus, we may assume that any partial maximum perfect $[1,2]$-factor of $G$, which we consider, does not contain even cycles.  
    
  Let $F$ be a partial maximum perfect $[1,2]$-factor (without even cycles) such that it minimizes the number of chords of $G$ among the edges in $E(F)$. By our assumptions, $F$ is not a perfect $[1,2]$-factor. Thus, there exists a vertex $v\in V(G)\setminus V(F)$.  Notice that if $v$ is not in a $2$-connected component containing at least three vertices,  then it is adjacent to some cut-vertex $u\in V(F)$. 
     Moreover, there exists a vertex $w\in N_F(u)$ such that $w$ and $v$ are in different components of $G-u$ and the claim follows. Hence, we assume without loss of generality that $v\in V(A_1)$.

    We denote by $Cut(F)\subseteq V(A_1)$ the set of cut-vertices of $G$ such that for any $u\in Cut(F)$ we have $u\in V(F)\cap V(A_1)$ and there exists a vertex $w\in N_F(u)\setminus V(A_1)$. In other words, $Cut(F)$ is the set of cut-vertices of $A_1$ which are also in $V(F)$ and adjacent in $F$ to a vertex outside of $A_1$.    
    
By our assumptions $A_1$ has a cut-vertex $c$, otherwise $G=A_1$ is a 2-connected graph that is Hamiltonian and thus admits a perfect $[1,2]$-factor.    Furthermore, we may choose $c$ such that $c\in Cut(F)$. Indeed, if such a vertex does not exist, then we could select $C(A_1)$ as a component of $F$ which is a contradiction on the maximality of $F$.
   
    We consider $c\in Cut(F)$ with a shortest distance to $v$ along the cycle $C(A_1)$ among the vertices in $Cut(F)$. Notice that we have a vertex $w\in N_F(c)\setminus V(A_1)$. Let $P_A$ be a shortest path along $C(A_1)$ from $v$ to $c$. Let us denote by $V(P_A)=\{v,v_1,\dots,v_h,c\}$ its vertices and by $E(P_A)=\{vv_1,v_1v_2,\dots,v_{h-1}v_h,v_hc\}$ its edges. Notice that if all vertex-pairs $\{v_{2i-1},v_{2i}\}$ form $K_2$-components in $F$, then we could instead consider $K_2$-components formed by vertex-pairs $\{v,v_1\}$ and $\{v_{2i},v_{2i+1}\}$. In this manner, if $h$ is even, then $v_h$ would be adjacent to $c$ and not in $F$ and $c$ would be a cut-factor vertex. While if $h$ is odd, then we would increase the size of $F$. Thus, we may assume that on the path $P_A$, there is a vertex $z\neq c$ such that $N_F(z)\not\subseteq V(P_A)$. Let $z$ be the first such vertex on the path $P_A$ from $v$ to $c$. Furthermore, let $N_F(z)=\{z_1,z_2\}$ (note that we can have $z_1=z_2$ if $N_F[z]$ induces a $K_2$).  Notice that since $c$ is the closest vertex in $Cut(F)$ to vertex $v$ along $C(A_1)$, we have $z\not\in Cut(F)$. Thus, $z_1,z_2\in V(A_1)$ and at least one of the edges $zz_1$ and $zz_2$ is a chord. In the following, we assume without loss of generality that $zz_1$ is a chord.
    Let us denote $z=v_e$. Notice that vertex pairs $v_{2i-1},v_{2i}$ for $2i<e$ form $K_2$-components in $F$. Consequently, as in the previous paragraph, we may ``\textit{shift}" $K_2$-components so that $v\in V(F)$ and $v_{e-1}\not\in F$. Thus, we may assume, without loss of generality, that $e=1$ and $v\in N_G(z)$. Assume first that $z$ belongs to a cycle $C_z$ in $F$. By our assumption on $F$, cycle $C_z$ has odd length. Hence, instead of $C_z$, we could have considered $K_2$-components in $F$ using vertices $\{v\}\cup V(C_z)$ which is against the maximality of $F$. Thus, $z$ belongs to a $K_2$-component in $F$ and $z_1=z_2$. However, now we could modify $F$ by removing $K_2$-component on vertices $z,z_1$ and by adding $K_2$-component containing $z$ and $v$. This is against $F$ containing a minimum number of chords of $G$. Thus, $z$ does not exist and the claim follows.
\end{proof}

Lemmas \ref{LemPVMP} and  \ref{LemHamiltonianCom} imply that if graph class $\mathcal{G}$ is closed under taking induced subgraphs and for any graph $G\in \mathcal{G}$ every $2$-connected component is Hamiltonian, then we have $o(G)=\D$ if and only if graph $G$ admits a perfect $[1,2]$-factor. In particular, we obtain following results for block and outerplanar graphs.

\begin{theorem}\label{TheBlock}
    Let $G$ be a block graph. We have $o(G)=\D$ if and only if $G$ admits a perfect $[1,2]$-factor. 
\end{theorem}
\begin{proof} The class of block graphs is closed by induced subgraphs. Moreover, each $2$-connected component of a block graph is a clique. Consequently, every $2$-connected component of size at least 3 is Hamiltonian. By Lemma \ref{LemHamiltonianCom}, block graphs have the cut-factor property and thus, by Lemma \ref{LemPVMP}, we obtain the equivalence.
\end{proof}

We note that if a block graph $G$ admits a perfect $[1,2]$-factor $F$, then it admits a perfect $[1,2]$-factor which has only $K_2$ and $K_3$-components. Such a factor can be seen as a pairing dominating set, which leads to the following corollary.

\begin{corollary}\label{corBlock}
Let $G$ be a block graph. Then $o(G)=\D$ if and only if $G$ admits a pairing dominating set. 
\end{corollary}

We next consider outerplanar graphs.

\begin{theorem}\label{TheOuterPlanar}
    Let $G$ be an outerplanar graph. We have $o(G)=\D$ if and only if $G$ admits a perfect $[1,2]$-factor. 
\end{theorem}
\begin{proof}
   Every 2-connected component on at least three vertices of an outerplanar graph has a Hamiltonian cycle \cite{chartrand1967planar}. Thus by Lemma \ref{LemHamiltonianCom}, outerplanar graphs have the cut-factor property.  Since they are closed by induced subgraphs, using Lemma \ref{LemPVMP} we obtain the equivalence.
\end{proof}

In the particular case of bipartite  graphs, if there is a perfect $[1,2]$-factor, there is a perfect matching. Therefore we have the following result.

\begin{corollary}
Let $G$ be a bipartite outerplanar graph. Then $o(G)=\D$ if and only if $G$ has a perfect matching.
\end{corollary}

The results of this section lead to polynomial algorithms for deciding who wins in  outerplanar and block graphs.

\begin{corollary}\label{CorOutBlockDecision}
   Given an outerplanar or a block graph $G$, \PBoutcomeD is polynomial. 
\end{corollary}
\begin{proof}
    Claim follows from Theorems \ref{TheBlock} and \ref{TheOuterPlanar} together with Corollary \ref{Cor12FactorPol}.   
\end{proof}

Note that the strategy for both players when Staller starts can be described in polynomial time. When Dominator wins, she just answer in each component of the perfect $[1,2]$-factor following her strategy in it. When Staller wins, he follows the strategy exhibited in Lemma \ref{LemPVMP}: he plays on a cut-factor vertex and follows the strategy on the component where Dominator did not play that does not have a perfect $[1,2]$-factor.

\section{Pairing strategies and application to interval graphs}\label{sec:interval}

In this section we focus $\mathcal F$-factors with $\mathcal F=\{K^+_{2,\ell}, \ell\geq 0\}$ that are equivalent to pairing dominating sets with adjacent pairs.
We first give some properties of pairing dominating sets when the pairs are adjacent. We then focus on interval graphs for which if there is a pairing dominating set, then there is an adjacent pairing dominating set.  Then we prove that an interval graph has outcome $\D$ if and only if it has a pairing dominating set. It implies in particular that the problem of deciding if an interval graph has outcome $\D$ is in $\mathsf{NP}$ (instead of $\mathsf{PSPACE}$). Finally, we discuss the complexity of finding a pairing dominating set in an interval graph.

\subsection{Adjacent pairing dominating sets}\label{subsec:adjacentPDS}

In Definition \ref{def:PDS} of pairing dominating sets, the pairs are not necessarily adjacent. When they are adjacent, we prove that a pairing dominating set is equivalent to a $\{K^+_{2,\ell}, \ell\geq 0\}$-factor.

\begin{proposition}\label{prop:PDSasfactors}
   Let $G=(V,E)$ be a graph. Then $G$ admits a pairing dominating set whose pairs are adjacent vertices if and only if $G$ has a  $\{K^+_{2,\ell}, \ell\geq 0\}$-factor.
\end{proposition}
Note that a $\{K^+_{2,\ell}, \ell\geq 0\}$-factor is also a $\{K_2,K_3,K^+_{2,\ell}, \ell\geq 3\}$-factor.

\begin{proof}
    Assume that $G$ has a PDS $S=\{(u_1,v_1),\dots ,(u_k,v_k)\}$ whose pairs are adjacent vertices.
    By definition, for each vertex $x$ of $G$, there exists a pair $(u_i,v_i)$ of $S$ such that $x$ is adjacent to both $u_i$ and $v_i$. Thus we can partition the vertices of $G$ in $k$ sets $X_1$,\dots ,$X_k$ such that every vertex in $X_i$ is adjacent to both $u_i$ and $v_i$. Since $u_i$ and $v_i$ are adjacent, we can assume that $X_i$ contain $u_i$ and $v_i$. Let $E_i=\{u_ix,v_ix | x\in X_i\} $. Note that the graph $(V_i,E_i)$ is either an edge (if $X_i=\{u_i,v_i\}$) or isomorphic to $K^+_{2,\ell}$.
    Let $H$ be the graph on vertex set $V$ with edges the union of sets $E_i$. Then $H$ is a $\{K^+_{2,\ell}, \ell\geq 0\}$-factor of $G$. 

    Let now $H$ be a $\{K^+_{2,\ell}, \ell\geq 0\}$-factor of $G$. Note that each connected component of $H$ admits two universal vertices. Let $S$ be the set of pairs constituted by the two universal vertices of each connect component. Then $S$ is a pairing dominating set whose pairs are adjacent.
\end{proof}

In the following, we will only consider {\em adjacent pairing dominating sets} (APDS in short) that are PDS where the pairs are adjacent. Deciding if a graph has an adjacent pairing dominating sets is NP-complete. Indeed the reduction in \cite{MBdomgame} for PDS can be easily adapted for APDS. 
However, we will prove that the property of having an APDS is definable in $\textsc{MSO}_2$ which implies in particular that it can be checked in linear time for graphs of bounded treewidth \cite{courcellemso}.
Informally, a graph property $\mathcal{P}$ is definable in $\textsc{MSO}_2$ if there exists 
formula $\varphi$ using quantifiers over vertices or edges, sets of edges or vertices and incidence relation $\texttt{inc}(u, e)$ such that, a graph $G$ satisfies $\mathcal{P}$ if and only if $\varphi$ is true for $G$ (see \cite{courcellemso} for a formal definition).

\begin{theorem}\cite{courcellemso}
Let $\mathcal{P}$ be a graph property definable in $\textsc{MSO}_2$. Then there exists a function $f$ such that deciding whether a graph $G$ satisfies $\mathcal{P}$ can be computed in time $f(t)n$ where $t$ is the treewidth of $G$ and $n$ is the number of vertices of $G$.
\end{theorem}

\begin{proposition}\label{prop:MSO}
Let $G$ be a graph. The property ``$G$ has an adjacent pairing dominating set" is definable in $\textsc{MSO}_2$.
\end{proposition}

\begin{proof}
Given a graph $G = (V,E)$, the following formula expresses the property for $G$ to have a pairing dominating set with adjacent pairs:
\begin{equation}\label{form:pair_dom}
    \exists S \subseteq E, (\forall v \in V, \exists e \in S, \texttt{dom}(e,v))\wedge \texttt{non\_adj}(S)
\end{equation}
where $\texttt{dom}(e,v)$ is defined by
\begin{align*}
    \texttt{dom}(e,v) :=& \texttt{inc}(v,e) \vee \exists u_1,u_2 \in V \Bigl(u_1 \neq u_2  \wedge \texttt{inc}(u_1,e)\\ 
    \wedge& \texttt{inc}(u_2,e) \wedge \texttt{edge}(u_1,v) \wedge \texttt{edge}(u_2,v)\Bigr)
\end{align*}
$\texttt{edge}(u, v)$ is defined by
$$\texttt{edge}(u,v) := u \neq v \wedge \exists e (\texttt{inc}(u,e) \wedge \texttt{inc}(v,e))$$ 
and where $\texttt{non\_adj}(S)$ is defined by
$$\texttt{non\_adj}(S) := \forall e_1,e_2 \in P, \exists u \in V, \texttt{inc}(u,e_1) \wedge \texttt{inc}(u,e_2) \implies e_1 = e_2.$$

The formula for $\texttt{dom}$ states that a vertex $v$ is dominated by the vertices incident to the edge $e$. Indeed, for a vertex to be dominated by these vertices, it needs to be in the closed neighborhood of both of them. Therefore, it is either incident to the edge or adjacent to both the incident vertices of the edge.

The formula for $\texttt{non\_adj}$ states that no vertex is incident to two different edges of $S$.

Therefore, Formula (\ref{form:pair_dom}) states that there exists a set of edges $S$ such that all the vertices of the graph are dominated by at least one of the edges of $S$ and no two edges of $S$ are incident to each other, which the definition of a pairing dominating set with adjacent pairs.
\end{proof}

Notice that a complete bipartite graph $K_{n,m}$ admits an adjacent pairing dominating set if and only if $n=m$.
Thus, we can reuse the proof of Proposition 5.13 in~\cite{courcellebook}  concerning hamiltonicity to show that the property
``$G$ has an adjacent pairing dominating set" is not definable in $\textsc{MSO}_1$ and consequently not in first order logic.

Note that when the pairs of a pairing dominating set are not adjacent, it is not clear that it can be written as a factor. Moreover, we do not know if having a general (i.e not necessarily adjacent) PDS can be written with an $\textsc{MSO}_2$-formula.

\subsection{Pairing dominating sets in interval graphs}\label{subsec:intervalPDS}

In all the rest of this section, interval graphs will be represented by a family of closed intervals of $\mathbb R$. For an interval $u$, we denote by $\min(u)$ (respectively $\max(u)$) the left endpoint or starting date (resp. right endpoint or ending date) of $u$. An interval representation of an interval graph $G$ is {\em proper} if no interval is included in another interval. A {\em unit} interval graph is an interval graph that has a representation where all the intervals have length 1. It is well known that an interval graph is a unit interval graph if and only if it has a proper representation.

We first prove that having a pairing dominating set or an adjacent pairing dominating set is equivalent in interval graphs.

\begin{proposition}
    \label{prop:adjacent}
    Let $G$ be an interval graph that has a pairing dominating set. Then $G$ has a pairing dominating set whose pairs are adjacent.
\end{proposition}

\begin{proof}
Let $G$ be an interval graph that has a pairing dominating set $S$. We choose $S$ that contains a minimal number of pairs that are not adjacent.
If there are only adjacent pairs in $S$, then we are done. 
Thus, assume that there exists a non-adjacent pair $(u,v)$ in $S$ and among all non-adjacent vertex pairs, choose vertex $u$ that minimises $\max(u)$. Since $u$ and $v$ are not adjacent, we have $\max(u)<\min(v)$. By minimality, the pair $(u,v)$ cannot be removed from $S$. This means that there exists an interval $w$ that is pair-covered only by $(u,v)$. We necessarily have the following inequalities: $\min (u) <\min(w)\leq \max(u)< \min(v) \leq\max (w)<\max(v)$. 
Indeed, the first and last inequalities follow from the fact that $u$ and $v$ must be pair-covered by pairs of $S$ distinct from $(u,v)$. If $w$ was containing $u$ or $v$, then it would be pair-covered by one of these pairs, which contradicts the fact that $w$ is only pair-covered by $(u,v)$. The other inequalities are just the translation that $w$ is intersecting $u$ and $v$.

Assume first that $w$ is not in a pair of $S$. Let $S'=(S\setminus \{(u,v)\}) \cup \{(u,w)\}$. We claim that $S'$ is still a PDS. Indeed, any vertex pair-covered by $(u,v)$ must contain $\max(u)$ and thus is intersecting $w$ and is pair-covered by $(u,w)$.

Assume next that $w$ is in a pair $(w,t)$ of $S$. Since $w$ is only pair-covered by $(u,v)$, vertices $w$ and $t$ are not adjacent. By minimality of $\max(u)$ among the non-adjacent pairs of $S$, we must have $t$ after $w$. In other words, $\max(w)<\min(t)$. Let $S'=(S\setminus \{(u,v),(t,w)\} )\cup \{(u,w),(v,t)\}$. We claim that $S'$ is a PDS. Indeed, as before, any vertex pair-covered by $(u,v)$ is pair-covered by $(u,w)$. Consider now a vertex $x$ pair-covered by $(w,t)$, it must contain $\max(w)$. Since $\min(v)\leq\max(w)<\max(v)$, interval $x$ is also intersecting $v$ and thus is pair-covered by $(v,t)$.

In both cases, we  exhibit a PDS $S'$ with less non-adjacent pairs, contradicting the minimality of $S$.
\end{proof}

Thus, by Proposition \ref{prop:PDSasfactors}, considering strategies using $\mathcal F$-factors or pairing dominating sets is equivalent in interval graphs. In the following we will rather use the pairing dominating set approach. We will always consider adjacent pairing dominating sets.


We now solve the case of unit interval graphs. In particular, having a PDS or having a $[1,2]$-factors are equivalent in unit interval graphs  to have outcome $\D$. 

\begin{theorem}
    \label{thm:unit_int}
    Let $G$ be a connected unit interval graph on $n\geq 2$ vertices. Let $v_1$,\dots ,$v_n$ be the vertices of $G$ indexed by increasing starting dates. 
    Then the following propositions are equivalent:
    \begin{enumerate}
        \item $G$ has outcome $\D$;
        \item $G$ has a $\{K_2,K_3\}$-factor;
        \item $G$ has a pairing dominating set;
        \item $n$ is even or some vertex $v_{2i}$ is not a cut-vertex.
    \end{enumerate}
\end{theorem}

Note that if a connected unit interval does not have outcome $\D$ then its outcome is $\mathcal N$. Indeed, Dominator wins as first by starting at $v_{2}$ and then pairing $v_{2i}$ with $v_{2i+1}$ for $i\geq 2$. All the vertices $v_{2i}$ must be cut-vertices. See Figure~\ref{fig:unit_interval} for a general illustration of these interval graphs.

\begin{proof}
   First note that since $G$ is connected and a unit interval graph, $v_i$ is adjacent to $v_{i+1}$ for all $i$. Moreover, since $K^+_{2,\ell}$ is not a unit interval graph for $\ell\geq 3$, by Propositions \ref{prop:PDSasfactors} and \ref{prop:adjacent} a unit interval graph has a pairing dominating set if and only if it has a $\{K_2,K_3\}$-factor. Thus (2) and (3) are equivalent. By Proposition \ref{prop:factorD}, (2) and (3) implies (1).

  We now prove that (4) implies (2). If $n$ is even, there is a perfect matching $M=\{(v_{2i-1},v_{2i}), i\in\{1,\dots ,n/2\}\}$ in $G$ and thus a $\{K_2,K_3\}$-factor.
    If $n$ is odd and $v_{2i}$ is not a cut-vertex for some $i$, it means that $v_{2i-1}$ and $v_{2i+1}$ are adjacent. Then there is a $\{K_2,K_3\}$-factor with edges $v_1v_2$ up to $v_{2i-3}v_{2i-2}$, triangle $v_{2i-1},v_{2i},v_{2i+1}$ and edges $v_{2i+2}v_{2i+3}$ up to $v_{n-1}v_n$.

    Finally, we prove that (1) implies (4) by contraposition: we assume that (4) is not true and prove that Staller has a winning strategy as first player. Since (4) is not true, $n$ is odd and all vertices $v_{2i}$ are cut-vertices. Then Staller starts by claiming $v_2$. Then Dominator must claim $v_1$ that is only adjacent to $v_2$. Staller goes on by claiming all the vertices $v_{2i}$ in increasing order. Each time, Dominator has to answer $v_{2i-1}$ that can only be dominated by itself. At the end, Staller can claim $v_{n}$ and wins. 
\end{proof}

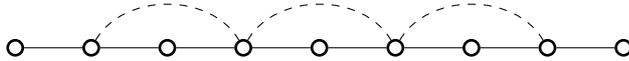
\begin{figure}[h]
    \centering
    \begin{tikzpicture}
        \node[noeud](a) at (0,0){};
        \node[noeud](b) at (1,0){};
        \node[noeud](c) at (2,0){};
        \node[noeud](d) at (3,0){};
        \node[noeud](e) at (4,0){};
        \node[noeud](f) at (5,0){};
        \node[noeud](g) at (6,0){};
        \node[noeud](h) at (7,0){};
        \node[noeud](i) at (8,0){};

        \draw (a)--(b)--(c)--(d)--(e)--(f)--(g)--(h)--(i);
        \draw[dashed] (b) to[out=60, in=120] (d);
        \draw[dashed] (d) to[out=60, in=120] (f);
        \draw[dashed] (f) to[out=60, in=120] (h);

        
    \end{tikzpicture}
    \caption{Representation of the unit intervals with nine vertices where Dominator cannot win in second. Dashed edges can be added or not in $G$. Any other edge added to $G$ will make the graph winning for Dominator in second position by Theorem \ref{thm:unit_int}.}
    \label{fig:unit_interval}
\end{figure}

\subsection{Equivalence between  PDS and outcome $\D$ in interval graphs}\label{subsec:eqinterval}

So far, in Subsection  \ref{subsec:intervalPDS}, we have shown that an interval graph admits a PDS if and only if it admits an APDS. Furthermore, in Subsection \ref{subsec:adjacentPDS}, we have proven that the property of a graph admitting an APDS can be written using an MSO-formula which implies that checking whether a graph on $n$ vertices admits an APDS can be done in $f(t)n$ time for some function $f$ and treewidth $t$. In this subsection, we continue by showing that for an interval graph $G$ we have $o(G)=\D$ if and only if graph $G$ admits a PDS (or equivalently, an APDS). In particular, this implies that the problem of checking whether $o(G)=\D$ is in {\sf NP} rather than in {\sf PSPACE} for interval graphs.

\begin{theorem}\label{thm:eqint}
Let $G$ be an interval graph. We have $o(G)=\D$ if and only if $G$ admits a pairing dominating set.
\end{theorem}

In the following, we present multiple lemmas that will be combined together to obtain the proof of the above result. This proof will be done by contradiction, by assuming that there exists an interval graph $G$ such that:
\begin{itemize}
    \item $o(G)=\D$;
    \item $G$ does not admit a pairing dominating set.
\end{itemize}

We assume that this counterexample $G$ satisfies the above conditions with minimum number of vertices and among them minimum number of edges. Then, the proof is organized in three major steps:
\begin{enumerate}
    \item We prove that $G - u$ admits a pairing dominating set where $u$ is the first interval to end (Lemmas \ref{LemmaU simlicial} and \ref{LemmaU is first}). The same property holds for $G-u'$ with $u'$ being the last interval to start.
    \item We then  consider the first interval $v$ for which Dominator has to answer on the left of $v$ if Staller starts by claiming $v$. We then show that $v$ is twin with another interval $x$ (Lemmas~\ref{lem:split} to~\ref{LemmavxPair}). 
    \item The end of the proof then considers the two following cases to show the contradiction, i.e. that $G$ admits a pairing dominating set: if $u$ (or equivalently $u'$) corresponds to one of the twins, then it is possible to adapt the pairing dominating set of $G - u$ to build a pairing dominating set of $G$ by simply doing an exchange with one of the twin intervals. In the other case, the graph $G$ will be partitioned into two non-empty parts that are winning for Dominator. The minimality of the counterexample ensures that each part admits a pairing dominating set. Then a general pairing dominating set of $G$ will be build from these two sets by adding the twin pair in the set.
\end{enumerate}

We say that a set of pairs of vertices $S$ {\em pair-covers} a vertex $x$ if there is a pair $(u,v)\in S$ such that $x$ is adjacent to both $u$ and $v$.
It is {\em valid} if all the pairs are disjoint and if the vertices inside a pair are adjacent.
Note that if $S$ is valid and pair-covers all vertices in $V$, then $S$ is an adjacent pairing dominating set.

We consider an interval representation of $G$ where no intervals start or finish in the same point.
Let $u$ be the interval in $G$ with the smallest $\max(u)$ such that there is a valid set of pairs of vertices of  $G$ that pair-covers all the vertices, called $Y'$, ending strictly after $\max(u)$. We further denote $Y=\{u\}\cup Y'$. 
By minimality of $u$ and since $G$ has no APDS, there are no valid sets of pairs that pair-cover all  vertices in the set $Y$.  In the following lemma, we first show that no intervals end between $\min(u)$ and $\max(u)$.

\begin{lemma}\label{LemmaU simlicial}
There are no intervals that end between $\min(u)$ and $\max(u)$. In particular, all the intervals intersecting $u$ contain $\max(u)$. Therefore, $u$ is simplicial (i.e. $N[u]$ is a clique).
\end{lemma}

\begin{proof}
We prove that there are no intervals ending between $\min(u)$ and $\max(u)$.
Assume by contradiction that there is such an interval $v$.
Consider a valid set of pairs $S$ that pair-covers all the vertices in $Y'$. We have $u\not\in S$ since $S$ does not pair-cover $u$. 
If $v$ is not in a pair of $S$, then $S\cup \{(u,v)\}$ pair-covers vertices in $Y$, a contradiction.
Thus, let $w$ be such that $(v,w)\in S$. If $w$ ends after $\min(u)$, then $(v,w)$ is covering $u$, a contradiction. Thus $w$ ends before $\min(u)$. Then the intervals that are covered by $v$ and $w$ are starting before $\min(u)$. Thus, $S'=(S\cup \{(u,v)\})\setminus \{(v,w)\}$ is valid and pair-covers every interval ending after $\max(u)$ and also $u$. Indeed, if an interval was ending after $\max(u)$ and was pair-covered by $(v,w)$ it is also pair-covered by $(v,u)$. Thus we obtain a contradiction. The claim follows by contradiction since now $w$ cannot start before or after $\min(u)$.
\end{proof}

We next show that no intervals end before interval $u$.

\begin{lemma}\label{LemmaU is first}
Interval $u$ is the first interval to end, that is, $\max(u)$ is smallest among all the vertices.
\end{lemma}

\begin{proof}
Let $G'=G[Y]$ be the graph $G$ that is induced by $Y$. We want to prove that $G=G'$. Assume by contradiction that it is not the case.

Graph $G'$ has no adjacent pairing dominating sets. Indeed, such a set of pairs would pair-cover $u$ and all the intervals ending after $\max(u)$, contradicting the definition of $u$.
We give a strategy for Dominator to win as the second player in $G'$, contradicting the minimality of $G$.

Let $\mathcal S$ be the strategy of Dominator in $G$. Dominator follows $\mathcal S$ in $G'$. When $\mathcal S$ asks Dominator to play outside of $G'$, Dominator claims any vertex of $G'$.
At the end, the only vertices of $G'$ that could not be dominated are the vertices in $N_{G'}[u]$. But $u$ is dominated in the game $G$. Since $N_G[u]=N_{G'}[u]$ by Lemma \ref{LemmaU simlicial}, vertex $u$ is also dominated in the game $G'$.
Thus, Dominator wins as the second player, a contradiction with the minimality of $G$. Hence, $G=G'$. In particular, there are no vertices that end before $\max(u)$, so $u$ is the first interval to end.\end{proof}

As a consequence, there is a valid set of pairs $S$ that pair-covers $V\setminus\{u\}$. Moreover, we can use symmetrical arguments for the last interval to start (denoted by $u'$).

\begin{corollary}\label{Corollary u' is last}
Let $u'$ be the last interval to start. Then, there is a valid set of pairs that pair-covers $V\setminus\{u'\}$.
\end{corollary}

Without loss of generality, we may assume that $u$ is the first interval to start (and to end by Lemma \ref{LemmaU is first}) and $u'$ is the last interval to end (and to start as in Corollary \ref{Corollary u' is last}).



%

In the proofs of following lemmas, we will often partition the intervals according to their rightmost points and use two different strategies for Dominator on each part. Before going into the details of these proofs, we give two general lemmas that will let us apply a strategy of Dominator on a small part of a graph and then merge strategies. Note that these lemmas are actually true for any graph. Following notation will be used in multiple following lemmas. Let $(V_1,V_2)$ be a partition of $G$. Let $V'_1$ be the vertices of $V_1$ that are adjacent to at least one vertex of $V_2$ and $V'_2$ be the vertices of $V_2$ that are adjacent to at least one vertex of $V_1$.

\begin{lemma}\label{lem:split} 
Assume that Dominator has a winning strategy from the position $(G,D,S)$ playing second. Then there exists a strategy for Dominator playing second to dominate all the vertices of $V_1\setminus V'_1$ using only unclaimed vertices of $V_1$.    
\end{lemma}
Note that vertices of $D$ that are in $V_2$ are not useful for Dominator to dominate the vertices of $V_1\setminus V'_1$.

\begin{proof}
    Assume by contradiction that this is not true. Then Staller playing first can isolate a vertex of $V_1\setminus V'_1$ with Dominator playing only vertices in $V_1$. But then Staller could win playing first in $(G,D,S)$ by applying this strategy. Indeed the vertex of $V_1\setminus V'_1$ that is isolated by Staller cannot be dominated by a move in $V_2$.
\end{proof}

\begin{lemma}\label{lem:union}
Consider the position $(G,D,S)$. Assume that Dominator has a strategy $\mathcal S_1$ playing second to dominate all the vertices of $V_1\setminus V'_1$ using only unclaimed vertices of $V_1$ and a strategy $\mathcal S_2$ playing second to dominate all the vertices of $V_2\setminus V'_2$ using only unclaimed vertices of $V_2$. Finally, assume that the vertices of $V'_1\cup V'_2$ are all dominated by a vertex in $D$. Then Dominator has a winning strategy in  $(G,D,S)$ playing second.   
\end{lemma}

\begin{proof}
When Staller claims a vertex in $V_i$, Dominator follows $\mathcal S_i$ and claims a vertex in $V_i$. The vertices of $V_i\setminus V_i'$ will be dominated by the strategy $\mathcal S_i$ whereas the vertices in $V'_1 \cup V'_2$ are dominated directly by $D$. 
\end{proof}

Let us denote by $G^{-v}$ the position $(G,\emptyset,\{v\})$  and by $G^{-v+w}$ the position $(G,\{w\},\{v\})$. Following lemma states that the first move of Dominator in $G$ can be adjacent to the move of Staller.

\begin{lemma}\label{LemmaAdjacencyPlay}
For any vertex $v\in V(G)$, there exists a winning strategy for Dominator from position $G^{-v+x}$ for some vertex $x\in N(v)$.
\end{lemma}
\begin{proof}
Let us assume on the contrary that there exists a vertex $v$, such that for each vertex $x\in N(v)$, there is a winning strategy for Staller from position $G^{-v+x}$. 
By Lemma \ref{BadPlacementLemma}, there is no interval $w$ such that $N[v]\subseteq N[w]$, otherwise Dominator could win from position $G^{-v+w}$. In particular, $v\notin \{u,u'\}$ (where $u$ and $u'$ are the first and last intervals of $G$).
Since Dominator is winning second in $G$, there exists a vertex $y\not\in N(v)$ such that there is a winning strategy for Dominator from position $G^{-v+y}$. 

Let us assume, without loss of generality, that $\min(y)>\max(v)$. Let $S_u$ be a pairing dominating set in $G\setminus \{u\}$. 
Let $h\in N(v)$ be a vertex chosen in the following way: 
\begin{enumerate}
    \item If $v$ is in a pair $(v,t)\in S_u$, 
    then we choose $h$ to be an interval which is not included in any other intervals and which contains $t$ (we can have $h=t$ if $t$ is not included in any intervals).
    \item If $v$ is not in any pair of $S_u$, then we choose $h\in N(v)$ such that $\max(h)>\max(v)$ and $h$ is not included in any interval. 
\end{enumerate}
Observe that $h$ is well-defined in the second case since graph $G$ is connected and $y$ starts after $\max(v)$. Moreover, $h$ is not included in any interval of $V(G)$. 
Next, we give a winning strategy for Dominator playing second from position  $G^{-v+h}$, which contradicts our assumption.

We partition $G$ using $h$. Let $V_1=\{w\in V(G)\mid \max(w)< \max(h)\}$ and $V_2=\{w\in V(G)\mid  \max(w)\geq\max(h)\}$. Let $V'_1$ and $V'_2$ be the sets of intervals in $V_1$ and $V_2$ adjacent to intervals of $V_2$ and $V_1$, respectively.
Note that all the intervals of $V'_1\cup V'_2$ are dominated by $h$. Indeed, an interval of $V'_2$ must start before $\max(h)$ to intersect an interval of $V'_1$. Moreover all the intervals of $V'_2$ must start after $\min(h)$ since $h$ is not included in any interval. Thus intervals of $V'_1$ are also dominated by $h$.

Since Dominator is winning in $G^{-v+y}$ playing second, by Lemma \ref{lem:split}, there is a strategy $\mathcal S_1$ for Dominator to dominate all the vertices of $V_1\setminus V'_1$ by claiming only vertices of $V_1\setminus \{v\}$. Note that $y\in V_2$ is not useful for $\mathcal S_1$.

Each vertex in $V_2\setminus V'_2$ is pair-covered by a pair of $S_u$ with only vertices of $V_2$. 
Indeed, if $z\in V_2\setminus V'_2$, then it cannot intersect and thus be dominated by a vertex of $V_1$. 
Let $\mathcal S_2$ be the strategy of Dominator following $S_u$ in $V_2$: whenever Staller claims a vertex of a pair of $S_u$ included in $V_2$, Dominator claims the other element of the pair.
Following this strategy, Dominator will dominate all the vertices of $V_2\setminus V'_2$ using only vertices of $V_2$.

We now apply Lemma \ref{lem:union} to merge $\mathcal S_1$ and $\mathcal S_2$ from position $G^{-v+h}$. Since $h$ dominates all the vertices of $V'_1\cup V_2'$, we can conclude that Dominator has a winning strategy playing second in $G^{-v+h}$, a contradiction.
%
\end{proof}

 From now on, we assume that in a winning strategy of graph $G$, Dominator always plays her first move adjacent to Staller's first move.

Let $v$ be the vertex with the smallest starting point from which the only adjacent vertices Dominator can continue to are on the left of $v$. More formally if Dominator has a winning strategy from position $G^{-v+x}$ for some vertex $x\in N(v)$, then $\max(x)<\max(v)$. We choose $x\in N(v)$ with the largest ending point such that Dominator has a winning strategy from $G^{-v+x}$.
Observe that vertex $v$ exists since Dominator can only answer on the left of $u'$ when Staller starts by claiming it.  Moreover, $v$ is not contained in any interval (otherwise, by Lemma \ref{BadPlacementLemma}, this interval would be a good answer for Dominator and will end after $v$). Interval $x$ is also not contained in any interval $t\neq v$, otherwise it would have been better for Dominator to claim $t$ instead of $x$, contradicting the maximality of $\max(x)$. 


\begin{lemma}\label{LemmavxPair}
There exists a winning strategy for Dominator from position $G^{-x+v}$.
\end{lemma}
\begin{proof}
Let us assume on the contrary that Dominator does not have a winning strategy from position $G^{-x+v}$. Hence, we have $\min(x)<\min(v)$. 
Since Dominator has a winning strategy in $G$ playing second, by Lemma \ref{LemmaAdjacencyPlay}, there exists an interval $f\in N(x)$ such that there exists a winning strategy for Dominator from position $G^{-x+f}$. Among these intervals, we choose $f$ with the largest endpoint. By the minimality of $\min(v)$, and since $\min(x)<\min(v)$, we have $\max(f)>\max(x)$ and $f$ is not included in any other interval.
Thus none of the three intervals $v$, $x$ or $f$ can be included one in another. We have two cases to consider.

Let us first consider the case where $\min(x)<\min(f)<\min(v)<\max(x)<\max(f)<\max(v)$. We construct a contradiction by showing that there is a winning strategy for Dominator from position $G^{-v+f}$, this is against the maximality of $x$. 
We partition $G$ according to $\max(f)$. Let $V_1=\{w\in V(G)\mid \max(w)< \max(f)\}$ and $V_2=\{w\in V(G)\mid  \max(w)\geq\max(f)\}$. 
As before, since $f$ is not included in any interval, one can check that $f$ intersects all the intervals of $V'_1\cup V'_2$ where $V'_1$ and $V'_2$ are defined as previously.
Since Dominator has a winning strategy playing second from position $G^{-x+f}$, by Lemma \ref{lem:split}, she has a strategy $\St_1$ to dominate $V_1\setminus V'_1$ using intervals in $V_1\setminus \{x\}$. Since she has a winning strategy playing second from position $G^{-v+x}$, by Lemma \ref{lem:split}, she has a strategy $\St_2$ to dominate $V_2\setminus V'_2$ using intervals in $V_2\setminus \{v\}$.
By Lemma \ref{lem:union}, using $\St_1$, $\St_2$ and since $f$ dominates all the vertices in $V'_1\cup V'_2$, Dominator has a winning strategy from position $G^{-v+f}$, contradicting the maximality of $x$. 

We now consider the case where $\min(x)<\min(v)<\min(f)<\max(x)<\max(v)<\max(f)$. We will construct the contradiction by showing that there is a winning strategy for Dominator from position $G^{-x+v}$. Let us partition the graph $G$ using interval $v$.  Let $V_1=\{w\in V(G)\mid \max(w)< \max(v)\}$ and $V_2=\{w\in V(G)\mid  \max(w)\geq\max(v)\}$. 
As before, since $v$ is not included in any interval, $v$ dominates all the intervals in $V'_1\cup V'_2$.
Using Lemma \ref{lem:split}, we define $\St_1$  from position $G^{-x+f}$ that will dominate $V_1\setminus V'_1$ using vertices of $V_1\setminus \{x\}$ and $\St_2$ from position $G^{-v+x}$ that will dominate $V_2\setminus V'_2$ using vertices of $V_2\setminus \{v\}$. Then using Lemma \ref{lem:union}, we obtain a strategy for Dominator from $G^{-x+v}$.

Together these cases cover all the possibilities for $f$ and hence, there is a winning strategy from position $G^{-x+v}$.
\end{proof}

We continue by studying the intersection of intervals $x$ and $v$.

\begin{lemma}\label{LemmavxIntersection}
The intersection of intervals $x$ and $v$ is not included in any other interval.
\end{lemma}
\begin{proof} 
Let us assume on the contrary that there exists an interval $f$ such that $\min(f)<\min(v)$ and $\max(f)>\max(x)$. Since $v$ nor $x$ are included in any other intervals and since $\min(x)<\min(v)$, we have $\min(x)<\min(f)<\min(v)<\max(x)<\max(f)<\max(v)$.  Moreover, we may assume that $f$ is not included in any interval.

We construct a contradiction by showing that there is a winning strategy for Dominator from position $G^{-v+f}$, this is against the maximality of $x$. 
As before, we partition the intervals using $f$: Let $V_1=\{w\in V(G)\mid \max(w)< \max(f)\}$ and $V_2=\{w\in V(G)\mid  \max(w)\geq\max(f)\}$. 
As before, since $f$ is not included in any interval, $f$ dominates all the intervals in $V'_1\cup V'_2$. Dominator wins from positions $G^{-v+x}$ and $G^{-x+v}$ by our assumption on $v$ and $x$ and by Lemma \ref{LemmavxPair}.
Using Lemma \ref{lem:split}, we define $\St_1$  from position $G^{-x+v}$ that will dominate $V_1\setminus V'_1$ using vertices of $V_1\setminus \{x\}$ and $\St_2$ from position $G^{-v+x}$ that will dominate $V_2\setminus V'_2$ using vertices of $V_2\setminus \{v\}$. Then using Lemma \ref{lem:union}, we obtain a strategy for Dominator from $G^{-v+f}$.
%
\end{proof}

We finally prove that $v$ and $x$ are twins, i.e. they have the same closed neighborhoods.

\begin{lemma}\label{LemmavxNeighborhood}
We have $N_G[v]=N_G[x]$.
\end{lemma}

\begin{proof} We first prove that no intervals start or end between $\max(x)$ and $\max(v)$. With similar arguments we could also prove that no intervals start or end between $\min(x)$ and $\min(v)$. Together, these prove that $N[v]=N[x]$.
 Let $G'$ be the graph obtained from $G$ by decreasing the value of $\max(v)$ until no interval starts between $\max(v)$ and $\max(x)$.  
 The graph $G'$, as a spanning subgraph of $G$, has no PDS. We will prove that Dominator wins in $G'$, thus by minimality of $G$, we will have $G=G'$, proving no intervals start or end between $\max(x)$ and $\max(v)$. 

We partition the graph according to $\max(x)$ but this time, $x$ is in $V_1$.
Let $V_1=\{w\in V(G)\mid \max(w)\leq \max(x)\}$ and $V_2=\{w\in V(G)\mid  \max(w)>\max(x)\}$. 
Using Lemma \ref{lem:split}, we define strategy $\St_1$  from position $G^{-x+v}$ that will dominate $V_1\setminus V'_1$ using vertices of $V_1\setminus \{x\}$ and strategy $\St_2$ from position $G^{-v+x}$ that will dominate $V_2\setminus V'_2$ using vertices of $V_2\setminus \{v\}$. 
Let $t\in V'_1$ be such that it intersects $z\in V'_2$. Then $\min(z)<\max(x)$ and thus $z$ contains $\max(x)$ and intersects both $v$ and $x$.  We must also have $\max(t)>\min(z)$. By Lemma \ref{LemmavxIntersection}, $z$ does not contain the intersection of $v$ and $x$ (if $z\neq v$). Hence, we must have $\min(z)\geq\min(v)$ where equality holds if and only if $z=v$. Together, it gives $\max(t)>\min(v)$ and thus $t$ intersects $v$ and $x$.
This means that both intervals $v$ and $x$ are dominated by all the intervals in $V'_1\cup V'_2$.
We now use Lemma \ref{lem:union} together with strategies $\St_1$, $\St_2$ and Dominator pairing $x$ with $v$ (i.e. she claims $v$ when Staller claims $x$ and vice versa), to obtain a strategy for Dominator from $G$. Note that in this strategy, the vertices of $V_2\setminus V'_2$ are dominated by vertices in $V_2\setminus \{v\}$. Hence this strategy is also a winning strategy for Dominator in $G'$ since the edges that have been removed are between $v$ and vertices in $V_2\setminus V'_2$. 

The proof that no  intervals start or end between $\min(x)$ and $\min(v)$ follows with symmetric arguments. Although there are two cases, one when $\min(x)<\min(v)$ and one when $\min(v)<\min(x)$. Otherwise, the proof is symmetric to the case with $\max(x)$ and $\max(v)$. Together these show that $N[x]=N[v]$.
\end{proof}

Using these lemmas, we are now ready to prove the main result of this subsection.

\begin{proof}[Proof of Theorem \ref{thm:eqint}]
Let $G$ be the minimal counterexample and $u$, $u'$, $v$ and $x$ be as we have previously defined them. 
By Lemma \ref{LemmavxNeighborhood}, we have $N_G[v]=N_G[x]$. Recall that neither $v$ nor $x$ can be included in any  intervals. 

Let us first consider the case where  $v=u'$ or $x=u$. Let us assume that $x=u$, the case with $v=u'$ is similar. 
Recall that, due to Lemma \ref{LemmaU is first}, we have a pairing dominating set $S_u$ in $G_u=G-u$. 
The vertex $v$ is pair-covered by some pair $(y_1,y_2)$. Since $u$ and $v$ are twins, $(y_1,y_2)$ is also pair-covering $u$, a contradiction.


Thus, we may assume that $x\not\in \{u,u'\}$ and $v\not\in\{u,u'\}$. Let $G_L$ be the graph induced by the vertices $w$ with $\max(w)\leq  \max(v)$ and $G_R$ be the graph induced by the vertices $w$ with $\min(w)\geq  \min(x)$.

Then Dominator has a winning strategy in $G_L$ and in $G_R$ playing second. Indeed, by symmetry, we only prove it for $G_L$.
Using Lemma \ref{superlemma}, we just need to prove that $(G_L,\{x\},\{v\})$ has outcome $\D$.
Let $V_1=V(G_L)$ and $V_2=V\setminus V_1$, i.e. $V_2$ contains the vertices that are finishing strictly after $\max(v)$. The vertices $V'_1$ that are in $V_1$ and adjacent to vertices in $V_2$ might intersect $v$ and $x$. We apply Lemma \ref{lem:split} with partition $(V_1,V_2)$ on $(G,\{x\},\{v\})$ which has outcome $D$ by Lemma \ref{superlemma}.
Then Dominator has a strategy to dominate $V_1\setminus V'_1$ using vertices of $V_1$. Thus she has a strategy to dominate all the vertices in $(G_L,\{x\},\{v\})$ since $x$ will dominate all $V'_1$.


 
 Moreover, we have $|V(G_L)|<|V(G)|$ and $|V(G_R)|<|V(G)|$ since $u\notin V(G_R)$ and $u'\notin V(G_L)$. Thus, by minimality of $G$, there is a pairing dominating set $S_L'$ for $G_L$ and $S_R'$ for $G_R$. 
Let $S_L$ be the pairs of $S'_L$ such that both intervals end before $\max(x)$ and $S_R$ be the pairs of $S'_R$ such that both intervals start after $\min(v)$.
We prove that $S=S_L\cup S_R\cup\{(v,x)\}$ is a pairing dominating set of the graph $G$, which leads to a contradiction.
Let $t$ be a interval. If $t\in N[x]=N[v]$ it is pairing dominated by $(v,x)$. By symmetry, we can assume that $\max(t)<\min(x)$. Interval $t$ must be pairing dominated by a pair $(w,z)$ in $S'_L$. Since $w$ and $z$ intersect $t$, they start before $\min(x)$. Since $x$ is not included in any interval, they must end before $\max(x)$ and thus $(w,z)\in S_L$ and is pairing dominated by $S$.
Thus $S$ is a pairing dominating set, a contradiction.
\end{proof}

An immediate consequence of Theorem \ref{thm:eqint} is that it gives a polynomial certificate for \PBoutcomeD in interval graphs.

\begin{corollary}
    \PBoutcomeD is in {\sf NP} when restricted to interval graphs.
\end{corollary}

\subsection{Computing PDS in interval graphs}\label{subsec:pdsinterval}

In this section, we discuss about the complexity of \PBPDS (i.e. deciding if a graph has a PDS) for interval graphs (which is equivalent to the complexity of \PBoutcomeD for this class of graphs by Theorem \ref{thm:eqint}). \PBPDS is {\sf NP}-complete, even if the pairs are adjacent \cite{MBdomgame} over all graphs. We did not manage to find a polynomial-time algorithm for \PBPDS in interval graphs nor to prove that it is {\sf NP}-complete. 
However, we present some positive algorithmic results for \PBPDS in interval graphs. Thanks to Propositions \ref{prop:MSO} and  \ref{prop:adjacent}, since \PBPDS can be expressed using a $MSO_2$ formula in interval graphs, it is $\mathsf{FPT}$ parameterized by the treewidth or equivalently in interval graphs by the clique number \cite{courcellemso}.

In the following, we extend this result to a less restrictive parameter defined as follows. A representation of an interval graph is {\em $k$-nested} if there are no chains of $k+1$ intervals included in each other. Let $G$ be an interval graph. We denote by $\nu(G)$ the smallest $k$ such that $G$ has a $k$-nested representation.
This parameter seems to have first appeared in \cite{nested-intro} under the name {\em depth}. It was formally defined and investigated in \cite{nested} where, in particular, it is proved that $\nu(G)$ can be computed in linear time on the number of vertices and edges. 
We always have $\nu(G)\leq \omega(G)$ where $\omega(G)$ is the clique number of $G$ (since nested intervals induce a clique).
Value $\nu(G)$ is also always smaller than or equal to the minimum number $\lambda(G)$ of different lengths of intervals one needs for representing $G$. Indeed, nested intervals must have different lengths. 

In \cite{nested-FO}, it is proved that checking an FO-formula for a $k$-nested intervals is FPT in $k$ and the length of the formula.  However, concerning the problem \PBPDS in interval graphs, it can only be expressed using MSO-formulas thus we cannot derive an FPT algorithm from the formula. 
We prove that \PBPDS is in {\sf XP} parameterized by $\nu(G)$.

\begin{theorem}\label{thm:xp}
    \PBPDS can be computed in time $O(n^{\nu(G)+3})$ in interval graphs. 
\end{theorem}

To prove this theorem, we use an important property of $k$-nested representations that is that they can be partitioned into $k$ proper interval representations \cite{nested}. Furthermore, the partitioning can be done in polynomial time. Indeed, consider applying Dilworth's Theorem \cite{dilworth} on the partial order ``$x<y$" if and only if $\min(x)<\min(y)$ and $\max(x)<\max(y)$. Indeed, an antichain in this partial order is exactly a set of  intervals contained in each other whereas the intervals in a chain form a proper interval representation. Finding such a partition can be done in $O(kn)$ \cite{Decomposition}.

We will use such a partition to provide a dynamic programming algorithm. Before going into the algorithm, we need to have a special representation of $G$ and define some notations.

Let $G$ be an interval graph on $n$ vertices. Let $k=\nu(G)$. Consider a $k$-nested representation $\mathcal R$ of $G$. We choose $\mathcal R$ so that no intervals have a common extremity and $\min(u)<\max(u)$ for each interval (this is always possible by eventually extending some intervals). Let $\mathcal X=\{X_1,\dots ,X_{k}\}$ be the partition of the intervals of $\mathcal R$ such that each part is a proper representation, i.e no intervals inside $X_i$ are included one in another. For each $X_i$, we order the intervals by their starting (or equivalently finishing) date and numerate them following order: $x_{i,1}<\dots <x_{i,n_i}$, which we call \textit{start-order}.

Let $t_1<\cdots<t_{2n}$ be the $2n$ extremities of all the intervals. Let $d_0,\dots,d_{2n}$ be real numbers such that $d_0<t_1<d_1<\cdots<d_{2n-1}<t_{2n}<d_{2n}$. Note that no intervals contain $d_0$ nor $d_{2n}$.

Let $S$ be an adjacent pairing dominating set of $G$. 
 We say that {\em $u$ is paired  after date $d$} if there exists an interval $v$ such that $(u,v)\in S$ and $d<\min(v)$.
Let $j\in \{1,\dots,2n-1\}$ and $i\in\{1,\dots,k\}$. 
We denote by $X(S,i,j)$ the intervals of $X_i$ that contain $d_j$ and are paired after $d_j$.

We say that $S$ is a {\em nice PDS related to $\mathcal X$} if it is an ADPS and for all $i,j$, either $X(S,i,j)$ is empty or is a set of consecutive intervals (with respect to the start-order) that contains the last interval of $X_i$ that contains $d_j$. In other words, if an interval is paired after $d_j$ then all the intervals after it that contain $d_j$ are also paired after $d_j$. Note that the definition of being nice depends on the representation $\mathcal R$ and the partition $\mathcal X$ of the intervals. Since the representation and the partition will not change through the proof, we will just use the term ``nice PDS".


\begin{lemma}\label{lem:nicePDS}
Graph $G$ admits an APDS of minimum size $d$ if and only if it admits a nice PDS of minimum size $d$. 
\end{lemma}

\begin{proof}
Recall that a nice PDS is also an APDS. Hence, the \textit{only if} part follows immediately.

Let us then consider the other direction. Take an APDS $S$ of $G$ with a minimum number of pairs 
and among them, take an APDS that maximizes the quantity $\sum_{(u,u')\in S} |u\cap u'|$.

We prove that $S$ is nice which will prove the lemma.
Assume by contradiction that $S$ is not nice. It means that there exist $i,j,s$ such that:
\begin{itemize}
    \item $x_{i,s}$ and $x_{i,s+1}$ contain $d_j$;
    \item there exists $v$ starting after $d_j$ such that $(x_{i,s},v)$ is a pair of $S$ and
    \item there is no interval $w$ starting after $d_j$ such that $(x_{i,s+1},w)$ is a pair of $S$.
\end{itemize}

We distinguish two cases. First, assume that $x_{i,s+1}$ is in no pair of $S$. Then let $S'$ be the set of pairs of $S$ where we have replaced the pair $(x_{i,s},v)$ by the pair $(x_{i,s},x_{i,s+1})$. The intersection of $x_{i,s}$ and $x_{i,s+1}$ strictly contains the intersection of $x_{i,s}$ and $v$. Thus $S'$ is still an APDS with the same number of pairs than $S$ but $\sum_{(u,u')\in S'} |u\cap u'|$ is larger than for $S$, a contradiction.

Assume next that $x_{i,s+1}$ appears in some pair $(x_{i,s+1},w)$ of $S$. By assumption, $w$ must start before $d_j$.
Note that $\max(v)>\max(w)$, otherwise one can remove the pair $(x_{i,s},v)$ from $S$, contradicting the minimality of $S$. 
If the intersection of $x_{i,s}$ and $x_{i,s+1}$ contains the intersections of the pairs $(x_{i,s},v)$ and $(x_{i,s+1},w)$, then one can remove the two pairs and add $(x_{i,s},x_{i,s+1})$, a contradiction with the minimality of $S$. 
Since the intersection of $x_{i,s}$ and $v$ is included in $x_{i,s}\cap x_{i,s+1}$ ($v$ starts after $\min(x_{i,s+1})$), we can assume that $w$ is finishing after $x_{i,s}$, i.e. $\max(w)>\max(x_{i,s})$. But then the intersection of $x_{i,s}$ and $v$ is included in the intersection of $x_{i,s+1}$ and $w$ and we can remove the pair $(x_{i,s},v)$ from $S$, contradicting the minimality of $S$.
\end{proof}

We will compute whether $G$ admits a nice pairing dominating set using a dynamic programming method and hence, by Lemma \ref{lem:nicePDS} and Proposition \ref{prop:adjacent}, whether $G$ admits a pairing dominating set. For that, we go through the dates $d_j$ and register which intervals will be paired after $d_j$ and until when the partial PDS is pair-covering vertices.
A {\em profile} is given by a $(k+2)$-tuple 
$\mathcal I=(d_j,s_1,\dots ,s_{k},t)$ with $d_j\in \{d_0,\dots ,d_{2n}\}$, $s_i\in \{1,\dots ,n_i+1\}$ and $t\in \{t_1,\dots ,t_{2n}\}$. Intuitively, $d_j$ is the date we are considering. The integers $s_i$ represent, for each $X_i$, the first interval of $X_i$ that contains $d_j$ and will be paired after $d_j$. Since we are looking for a nice PDS, it means that all the intervals of $X_i$ containing $d_j$ that are after $x_{i,s_i}$ will be paired after $d_j$. We put $s_i=n_i+1$ if there is no such interval. Finally, $t \in \{t_1,\dots ,t_{2n}\}$ represents the last date pair-covered by a pair in the partial PDS ($t=t_1$ means that there are no pairs in the PDS). 

Given a profile $\mathcal I$ and a date $d_j$, we define the set of intervals $X(\mathcal I,d_j)$ as follows: for each $i$, $x_{i,s}\in X(\mathcal I,d_j)$ if it contains $d_j$ and $s\geq s_i$. Intuitively, the set $X(\mathcal I,d_j)$ corresponds to the intervals containing $d_j$ that will be paired after $d_j$.

 We say that a profile is {\em valid} if:
\begin{enumerate}
    \item[(V1)] for each $i$, either $s_i=n_i+1$ or $x_{i,s_i}$ contains $d_j$;
    \item[(V2)] there exists a set $S$ (possibly empty) of pairs of adjacent intervals $\{(u_1,v_1),\dots ,(u_m,v_m)\}$ such that:
    \begin{enumerate}
    \item all the intervals appearing in $S$ are distinct; 
    \item all the intervals $u_i$ and $v_i$ start strictly before $d_j$; 
        \item no intervals $u_i$ or $v_i$ are in $X(\mathcal I,d_j)$;    
        \item if $S$ is empty, then $t=t_1$. Otherwise, $t=\max_{i=1}^m(\max(u_i\cap v_i))$, i.e. $t$ is the largest end of the intersections $u_i\cap v_i$;
        \item for any interval $x$ that ends before $d_j$, there exists a pair $(u_i,v_i)$ such that $x$ intersects both $u_i$ and $v_i$
        (we have a partial adjacent pairing dominating set that pair-covers all the intervals that are finishing before
        $d_j$).
    \end{enumerate}
\end{enumerate}

Note that we can have $t<d_j$.
The only valid profile for $d_0$ is $\mathcal I=(d_0,n_1+1,\dots ,n_{k}+1,t_1)$: there are no elements in $X(\mathcal I,d_0)$ and no (non-empty) partial pairing dominating set.

Let $j\geq 0$. We now explain how to compute the valid profiles at date $d_{j+1}$ from the profiles at date $d_j$. Let $\mathcal I=(d_j,s_1,\dots,s_{k},t)$ and $\mathcal I'=(d_{j+1},s'_1,\dots ,s'_{k},t')$ be two profiles at dates $d_j$ and $d_{j+1}$. We define {\em compatibility} between profiles as follows.
Assume first that there is an interval $u$ that starts between $d_j$ and $d_{j+1}$. In particular, $u$ contains $d_{j+1}$. 
Assume that $u\in X_{i_0}$ for some $i_0$ and let $u=x_{i_0,s}$. Furthermore, let $v=x_{i_1,s_{i_1}}\in X(\mathcal{I},d_j)$ be the first interval to end in $X(\mathcal{I},d_j)$.
We say that profiles $\mathcal I$ and $\mathcal I'$ are {\em add-compatible} if we are in one of the following cases:

\begin{enumerate}
    \item[(A1)] $t'=t$, $s'_i=s_i$ for all $i\neq i_0$,  $s'_{i_0}=s_{i_0}$ if $s_{i_0}\neq n_{i_0}+1$ and $s'_{i_0}=s$ if $s_{i_0}= n_{i_0}+1$.
(Intuitively, this case  corresponds to the case where $u$ must be paired after $d_{j+1}$.)
    \item[(A2)] Or, $t'=t$, $s'_i=s_i$ for all $i$  and $s'_{i_0}=s_{i_0}=n_{i_0}+1$. (This corresponds to the case where $u$ will not be in a pair. In particular, no elements of $X_{i_0}$ containing $d_{j}$ or $d_{j+1}$ can be paired after $d_j$ or $d_{j+1}$.)
    \item[(A3)] Or, $t'=\max(t,\max(u\cap v))$, $s'_{i_0}=n_{i_0}+1$, $s'_i=s_i$ for each $i\neq i_0,i_1$ and either  $s'_{i_1}=n_{i_1}+1$ if $X_{i_1}\cap X(\mathcal I,d_j)=\{v\}$ or otherwise $s'_{i_1}=s_{i_1}+1$.
        (This corresponds to the case where $u$ is paired with an interval in $X(\mathcal I,d_j)$, we will explain later why it can be assumed that $u$ is paired with $v$.)
\end{enumerate}

Note that there are at most three profiles $\mathcal I'$ that are add-compatible with $\mathcal I$.

Assume now that there is an interval $u$ that ends between $d_j$ and $d_{j+1}$. 
Then we say that $\mathcal I$ and $\mathcal I'$ are {\em remove-compatible} if all the following conditions are true:
\begin{enumerate}
    \item[(R1)] $u\notin X(\mathcal I,d_j)$;
    \item[(R2)] $s_i=s'_i$ for all $i$;
    \item[(R3)] $t=t'$;
    \item[(R4)] $\min(u)<t$.
\end{enumerate}

Note that there is at most one  remove-compatible profile with $\mathcal I$. 
Next lemma ensures that the definition of validity is consistent with that of compatibility.
\begin{lemma}\label{lem:validcomp}
If  $\mathcal I=(d_j,s_1,\dots,s_k,t)$ and $\mathcal I'=(d_{j+1},s_1',\dots,s_k',t')$ are compatible and $\mathcal I$ is valid, then $\mathcal I'$ is valid.
\end{lemma}

\begin{proof}
Consider $\mathcal I=(d_j,s_1,\dots,s_k,t)$ and $\mathcal I'=(d_{j+1},s_1',\dots,s_k',t')$. Assume that $\mathcal I$ is valid.

Assume first that an interval $u=x_{i_0,s}$ starts between $d_j$ and $d_{j+1}$ and that $\mathcal I$ and $\mathcal I'$ are add-compatible. Let $v=x_{i_1,s_{i_1}}\in X(\mathcal{I},d_j)$ be the first interval to end in $X(\mathcal{I},d_j)$.
We prove that $\mathcal I'$ is valid. First note that (V1) is always satisfied.
We just have to check the condition when $s'_i\neq n_i+1$. When $s'_i=s_i$, we know from the fact that $\mathcal I$ is valid that $x_{i,s_i}$ contains $d_j$. Since no interval  finishes between $d_j$ and $d_{j+1}$, $x_{i,s_i}$ also contains $d_{j+1}$. When $s'_{i_0}=s$, $x_{i_0,s}=u$ clearly contains $d_{j+1}$. Finally, in the case (A3), we might have $s'_{i_1}=s_{i_1}+1$. But in this case, $X_{i_1}\cap X(\mathcal I,d_j)$ is not reduced to $v=x_{i_1,s_{i_1}}$. In particular, $x_{i_1,s_{i_1}+1}$ must be in the intersection which means that it contains $d_j$ and also $d_{j+1}$ since it cannot end between the two dates.

Let us now consider Point (V2) of definition of validity. If we are in the two first cases (A1) or (A2) of add-compatibility, then any partial PDS that validates $\mathcal I$ will validate $\mathcal I'$. 
In case (A3), let $S=\{(u_1,v_1),\dots,(u_m,v_m) \}$ be a partial pairing dominating set that validates $\mathcal I$. Then one can check that $S'=S\cup \{(u,v)\}$ validates $\mathcal I'$. Indeed, $v$ does not appear in $S$ by condition (V2.c), thus all the intervals of $S'$ are distinct and start strictly before $d_{j+1}$. Hence, conditions (V2.a) and (V2.b) are satisfied. Since neither $u$ nor $v$ is paired after $d_{j+1}$, vertices $u$ nor $v$ are not in $X(\mathcal I',d_{j+1})$ and condition (V2.c) is satisfied.
The value  $t'=\max(t,u\cap v)$ satisfies (V2.d) and finally, the set of intervals finishing before $d_{j+1}$ is the same as the set of intervals finishing before $d_j$. Since $S$  pair-covers them, so does $S'$  satisfying (V2.e).

Assume next that an interval $u=x_{i_0,s}$ ends between $d_j$ and $d_{j+1}$. First of all, Point (V1) of validity is still satisfied. Indeed, consider first each $i\neq i_0$. By Point (R2),  we have $s'_i=s_i$. Thus, either $s_i'=n_i+1$ or $s_i< n_i+1$ and $x_{i,s_i'}=x_{i,s_i}$ contains $d_{j}$. Since $i\neq i_0$, $x_{i,s_i'}$ also contains $d_{j+1}$. Consider then $i=i_0$. By Point (R1) of remove-compatibility, we have $u\notin X(\mathcal I,d_j)$. Thus, $s_{i_0}>s$ and hence, either $s_{i_0}=n_{i_0}+1$ or $x_{i,s_{i_0}'}=x_{i,s_{i_0}}$ and the claim follows as in the previous case.


Consider next Point (V2). Then, any partial pairing dominating set $S$ that validates $\mathcal I$ will validate $\mathcal I'$.
Indeed, Conditions (V2.a) to (V2.d) are clearly still true. Furthermore, vertex $u$ has also to be pair-covered by $S$.
Note that $u$ starts before $t$ (Condition (R4)) and ends after $d_j$. Let $(u_i,v_i)\in S$ such that $\max(u_i\cap v_i)=t$ (it exists since $t>\min(u)\geq t_1$). Either $t<d_j$ and then clearly $t\in u$.
Or $t>d_j$ but then $d_j$ is contained in $u_i$, $v_i$ and $u$. In both cases, $u$ intersects both $u_i$ and $v_i$ and thus is pair-covered by~$S$.
\end{proof}

The idea of our algorithm is to compute valid profiles starting by the unique valid profile at $d_0$ and then computing all the compatible profiles from it. To prove that it is correct, we must prove that if there is a PDS in $G$, then there must be a valid profile at $d_{2n}$ obtained this way. 

Consider a graph $G$ together with a nice APDS $S$ of minimum size. We define profile $\mathcal I_j^S=(d_j,s_1,\dots ,s_{k},t)$ as follows. Let $S_j\subseteq S$ be the pairs of $S$ for which both intervals start before $d_j$. 
 Let $i\in \{1,\dots ,k\}$. As in the definition of a nice PDS, let $X(S,i,j)$ be the intervals of $X_i$ that contain $d_j$ and that are paired after $d_j$. If $X(S,i,j)$ is non-empty, then we let $s_i$ be the smallest $s$ such that $x_{i,s}\in X(S,i,j)$. Otherwise, let $s_i=n_i+1$. Finally, let $t=\max_{(u,v)\in S_j}\{\max(u\cap v)\}$. Note that we have exactly $X(\mathcal I_j^S,d_j)=\cup_{i} X(S,i,j)$.


\begin{lemma}\label{lem:validPDS}
Let $G$ be an interval graph with a nice APDS $S$ of minimum size. Profiles $\mathcal I_j^S$ and $\mathcal I_{j+1}^S$ are compatible and valid for each $j\in \{0,\dots,2n-1\}$.
\end{lemma}

\begin{proof}
    
Let $j\in \{0,\dots ,2n\}$.
Consider $\mathcal I^S_j=(d_j,s_1,\dots ,s_{k},t)$ as they are defined above. 
Observe that $\mathcal I^S_0$ is valid (since 
$\mathcal I_0^S=(d_0,n_1+1,\dots ,n_{k}+1,t_1)$). Hence, if each pair $(I_j^S,I_{j+1}^S)$ is compatible, then all profiles are valid by Lemma \ref{lem:validcomp}.


We next prove that $\mathcal I_j^S$ and $\mathcal I_{j+1}^S$ are compatible for any $0\leq j<2n$.
Assume first that an interval $u=x_{i_0,s}$ starts between dates $d_j$ and $d_{j+1}$.
\begin{enumerate}
    \item Assume first that $u$ is in a pair $(u,v)$ of $S$ with $v$ that starts after $d_{j+1}$. Then $u\in X(S,i_0,j+1)$. Since $S$ is a nice PDS, the intervals in $X(S,i_0,j)$ and $X(S,i_0,j+1)$ must be consecutive. Moreover, we have $X(S,i_0,j+1)=X(S,i_0,j)\cup\{u\}$. This means that either $s'_{i_0}=s_{i_0}$ (if $X(S,i_0,j)$ is non-empty) or $s_{i_0}=n_{i_0}+1$ and  $s'_{i_0}=s$ (if $X(S,i_0,j)$ is empty). For the other $s'_i$, nothing has changed and thus $s'_i=s_i$. 
    \item Assume now that $u$ is not in a pair of $S$. Then $\mathcal I_j^S$ and $\mathcal I_{j+1}^S$ are defined similarly. Note that since $S$ is a nice PDS and since $u\notin X(S,i_0,j+1)$, set $X(S,i_0,j+1)$ must be empty. This implies that $X(S,i_0,j)$ is also empty. In particular $s_{i_0}=s'_{i_0}=n_{i_0}+1$.
    \item Assume now that $u$ is in a pair $(u,v)$ in $S$ with $v$ that starts before $d_j$. Since $u$ and $v$ must intersect, $v$ contains $d_{j}$.
    We must have $v\in X(S,i_1,j)\subseteq X(\mathcal I_j^S,d_j)$ for some $i_1$. Assume $v$ is not the interval of $X(\mathcal I_j^S,d_j)$ finishing first. Instead, let $v'$ be that interval. It must be in a pair $(v',u')$ with $u'\neq u$ starting after $d_{j+1}$. 
    If $v'$ ends before $u$ ends, then $v'\cap u' \subset v\cap u$ and $S$ is not minimal. If $v'$ ends after $u$, then $v\cap v'$ contains $v\cap u$ and $v'\cap u'$. Again $S$ is not minimal, a contradiction.
    Thus $v$ is defined as in the definition of add-compatibility, and the values of $s_i$ and $s'_i$ are defined in the same way.
\end{enumerate}

In all the cases, $\mathcal I_j^S$ and $\mathcal I_{j+1}^S$ are add-compatible.

Assume now that an interval $u$ is finishing between $d_j$ and $d_{j+1}$. It must be covered by a pair $(u',v')$ of $S$. Then $\min(u)< \max(u'\cap v')$ and $\min(u'\cap v')< \max(u) < d_{j+1}$. In particular, $u'$ and $v'$ are starting before $d_{j+1}$ and thus $d_j$ and are in $S_j=S_{j+1}$. Thus $t=t'\geq \max(u'\cap v')$ and $\min(u)< t$. It proves that  $\mathcal I_j^S$ and $\mathcal I_{j+1}^S$ are remove-compatible. 
\end{proof}

\begin{proof}[Proof of Theorem \ref{thm:xp}]
The algorithm  starts with the unique valid profile at $d_0$: $\mathcal I_0^S =(d_0,n_1+1,\dots ,n_{k}+1,t_1)$. Then for each date $d_i$, starting with $d_1$, we compute all the profiles that are compatible with the profiles at date $d_{i-1}$. By Lemma \ref{lem:validcomp}, all these profiles will be valid.
If the algorithm finds a valid profile $\mathcal I_{2n}$ at date $d_{2n}$, then the set $S$ that validates $\mathcal I_{2n}$ must be a PDS since all the intervals finish before $d_{2n}$.

For the other direction, if $G$ has a PDS, then by Proposition \ref{prop:adjacent} and Lemma \ref{lem:nicePDS}, graph $G$ has an APDS $S$ of minimal size that is nice. By Lemma \ref{lem:validPDS}, the algorithm must find all the profiles $\mathcal I_j^S$ defined before Lemma \ref{lem:validPDS}. 
In particular, it must find $\mathcal I_{2n}^S$ and return that $G$ has a PDS.

About the complexity, finding the value of $k=\nu(G)$ a $k$-nested representation can be done in time $O(n+m)$ \cite{nested}. From this representation, finding the partition into $k$ disjoint proper representations can be done (with the help of Dilworth's theorem, as we have mentioned) in time $O(kn)$ \cite{Decomposition}. Then, there are at most $n^{k+1}$ valid profiles at some date. Computing all the profiles that are compatible between two consecutive dates, takes linear time in $n$. Since we consider profiles at $2n+1$ different dates, the algorithm runs in total in $O(n^2+kn+n^{k+3})=O(n^{k+3})$ 
time.
\end{proof}

Since computing if interval graph $G$ admits a PDS or has outcome $\mathcal D$ is equivalent by Theorem~\ref{thm:eqint}, we obtain the following corollary:

\begin{corollary}
      \PBoutcomeD can be computed in time $O(n^{\nu(G)+3})$ in interval graphs.
\end{corollary}

As a conclusion to this section, recall that the general (non-parameterized) complexity of \PBPDS in interval graphs remains open:

\begin{open}\label{oq:pds}
    What is the complexity of \PBPDS in interval graphs?
\end{open}

\section{Conclusion and perspectives}

Although a significant step has been made in this paper about the resolution of the domination game on interval graphs, the complexity of the problem for general interval graphs remains a relevant open problem. In particular, we have shown that, in interval graphs, \PBPDS is in {\sf NP} instead of being {\sf PSPACE}-complete, leading 
to the structural Open problem~\ref{oq:pds} about the complexity of \PBPDS in interval graphs. Moreover, the equivalence between \PBoutcomeD and finding perfect $[1,2]$-factors for outerplanar graphs opens the door for the study of related classes of graphs such as graphs of treewidth 2.

In addition, the surprisingly simple resolution of regular graphs would naturally lead to the case of particular subgraphs of them. More precisely, the case of subcubic graphs seems to be the first step to consider towards this direction. As shown by Figure~\ref{fig:delta+1}, their characterisation will be more tricky with the two possible outcomes that may arise. In correlation with this examination of the degrees of the graph, as all known hardness reductions require graphs with high maximum degree, one can wonder whether there could be positive results when this maximum degree is bounded.

Finally, because of the structural motivations of the current work, only the outcome $\D$ was considered here. One can wonder if, for some families of graphs, a refinement can be done to have a full characterisation of the two other cases (i.e. $\St$ and $\mathcal{N}$). In particular, for the outcome $\mathcal{N}$, if a graph can be partitioned with a single star and a perfect $[1,2]$-factor, or with a star and a $K^+_{2,\ell}$-factor, then
Dominator can win playing first. 
 This is true since a star is itself an $\mathcal{N}$-graph. In particular if, in addition, the graph does not have outcome $\D$, then it has outcome $\mathcal{N}$. As for the \PBoutcomeD problem, one can investigate the other direction: if a graph has outcome $\mathcal{N}$, does it admit such a partition?
 As we already know from~\cite{MBdomgame} that if a tree is $\mathcal{N}$, then it can be decomposed with a star and a perfect matching, this work could be extended to some of the classes of graphs considered in the current paper.

 \subsection*{Acknowledgements}
This research was supported by the ANR project P-GASE (ANR-21-CE48-0001-01). Tuomo Lehtil\"a's research was supported by the Finnish Cultural Foundation and by the Academy of Finland grant 338797. Some parts of this research were carried out when he was in Univ Lyon, CNRS, INSA Lyon, UCBL, Centrale Lyon, Univ Lyon 2, LIRIS, UMR5205, F-69622, Villeurbanne, France and in University of Helsinki, department of computer science, Helsinki, Finland.

\bibliography{references}

\end{document}